\documentclass[12pt,reqno]{article}
\usepackage[utf8]{inputenc}
\usepackage{enumerate}
\usepackage{lmodern}
\usepackage[T1]{fontenc}
\usepackage{verbatim}
\usepackage{textcomp}
\usepackage{tikz}

\usepackage[english]{babel}
\usepackage[a4paper,vmargin={3.5cm,3.5cm},hmargin={2.5cm,2.5cm}]{geometry}
\usepackage[font=sf, labelfont={sf,bf}, margin=1cm]{caption}
\usepackage[pdftex]{hyperref}
\usepackage{xcolor}
\usepackage{amsmath,amsfonts,amssymb,amsthm,mathrsfs,mathtools,bbm}
\usepackage{empheq}
\newtheorem{theorem}{Theorem}[section]

\newtheorem{lemma}[theorem]{Lemma}
\newtheorem{proposition}[theorem]{Proposition}

\theoremstyle{remark}
\newtheorem{remark}[theorem]{Remark}
\theoremstyle{definition}
\newtheorem{definition}[theorem]{Definition}

\numberwithin{equation}{section}

\newcommand{\Po}{{\cal P}}
\newcommand{\cV}{{\cal V}}
\newcommand{\cC}{{\cal C}}
\newcommand{\tcC}{\tilde{{\cal C}}}
\newcommand{\cH}{{\cal H}}
\newcommand{\X}{{\cal X}}
\newcommand{\cX}{{\cal X}}
\newcommand{\cY}{{\cal Y}}
\newcommand{\la}{\lambda}
\newcommand{\R}{\mathbb{R}}
\newcommand{\RR}{\mathbb{R}}
\newcommand{\N}{\mathbb{N}}
\newcommand{\NN}{\mathbb{N}}
\newcommand{\E}{\mathbb{E}}
\newcommand{\EE}{\mathbb{E}}
\newcommand{\Z}{\mathbb{Z}}
\newcommand{\LL}{\mathscr L}
\newcommand{\MM}{\mathscr M}
\newcommand{\cE}{\mathscr E}

\renewcommand{\Pr}{\mathbb{P}}
\newcommand{\PP}{\mathbb{P}}
\newcommand{\bx}{{\bf x}}
\newcommand{\bX}{{\bf X}}
\newcommand{\bY}{{\bf Y}}
\newcommand{\by}{{\bf y}}
\newcommand{\bz}{{\bf z}}

\newcommand{\vol}{{\rm Vol}}
\newcommand{\diam}{{\rm Diam}}
\renewcommand{\emptyset}{\varnothing}
\newcommand{\cF}{\mathcal{F}}
\newcommand{\cN}{\mathcal{N}}
\newcommand{\tX}{\tilde{X}}

 \newcommand{\tocc}{\overset{c.c.}\longrightarrow}
 \newcommand{\toas}{\overset{a.s.}\longrightarrow}
 \newcommand{\toP}{\overset{\Pr}\longrightarrow}
 \newcommand{\toD}{\overset{{\cal D}}\longrightarrow}

\newcommand{\fmax}{f_{\rm max}}

\newcommand{\dtv}{d_{\mathrm{TV}}}
\newcommand{\dk}{d_{\mathrm{K}}}
\newcommand{\dw}{d_{\mathrm{W}}}

\newcommand{\de}{\delta}

\newcommand{\ep}{\varepsilon}
\newcommand{\eps}{\varepsilon}
\newcommand{\ph}{\varphi}
\newcommand{\Var}{\mathbb{V}\mathrm{ar}}
\newcommand{\1}{\mathbbm{1}}
\newcommand{\bN}{\mathbf{N}}

\newcommand{\bea}{\begin{eqnarray}}
\newcommand{\eea}{\end{eqnarray}}
\newcommand{\bean}{\begin{eqnarray*}}
\newcommand{\eean}{\end{eqnarray*}}

\DeclareMathOperator\dist{dist}

\newcommand{\Bin}{\mathrm{Bin}}

\newcommand{\mynegspace}{\hspace{-0.12em}}
\newcommand{\lvvvert}{\rvert\mynegspace\rvert\mynegspace\rvert}
\newcommand{\rvvvert}{\rvert\mynegspace\rvert\mynegspace\rvert}

\begin{document}
\title{\bf On $k$-clusters of high-intensity random geometric graphs
}
\author{Mathew D. Penrose \thanks{Department of
Mathematical Sciences, University of Bath, Bath BA2 7AY, United
Kingdom: {\texttt m.d.penrose@bath.ac.uk}
}
  \and Xiaochuan Yang
  \thanks{Department of Mathematics, Brunel University London, Uxbridge, UB8 3PH, United Kingdom: {\texttt xiaochuan.yang@brunel.ac.uk}
  {\texttt https://orcid.org/0000-0003-2435-4615}
  }
\\
}




\date{\today}

\maketitle

\begin{abstract}   
	Let $k,d $ be positive integers.  We determine a sequence of constants that are asymptotic to the probability that the cluster at the origin in a $d$-dimensional Poisson Boolean model with balls of fixed radius is of order $k$, as the intensity becomes large.  Using this, we  determine the  asymptotics of the mean of the number of components of order $k$, denoted $S_{n,k}$ in a random geometric graph on $n$ uniformly distributed vertices in a smoothly bounded compact region of $R^d$, with distance parameter $r(n)$ chosen so that the expected degree grows slowly as $n$ becomes large (the so-called mildly dense limiting regime). We also show that the variance of $S_{n,k}$ is asymptotic to its mean, and prove Poisson and normal approximation results for $S_{n,k}$ in this limiting regime. We provide analogous results for the corresponding Poisson process (i.e. with a Poisson number of points).

	 We also give similar results in the so-called mildly sparse limiting regime where $r(n)$ is chosen so the expected degree decays slowly to zero as $n $ becomes large.
	\\
\end{abstract}

\section{Overview}
\label{secintro}

\subsection{Introduction and background}
\label{subsecintro}

The {\em Poisson blob model} (PBM) of overlapping spheres centred on
random points in space, also known as the {\em Gilbert graph},
is perhaps the simplest model of random
clustering in a continuous space. Cluster formation is governed
entirely by spatial proximity of particles, and their locations
are governed by complete spatial randomness with no interactions, i.e. a 
homogeneous spatial Poisson process.
The {\em random geometric graph} (RGG) is obtained by restricting the PBM to
a finite window, or in an alternate version, by specifying the number of particles in the window.

In the present paper we investigate these models at high intensity, i.e. with
a  high density of particles relative to the range at which connections form.  
For high densities, one expects most of the particles to lie in a single giant cluster, but the spatial randomness means that from time to time smaller
clusters may also be seen. We consider here clusters of fixed order $k$.
We provide precise asymptotics on the frequency of these clusters
as the intensity increases, and moreover describe the fluctuations of the number of clusters in a window that grows in size simultaneously
with the growth of the intensity; the window needs to be large enough to
observe clusters of order $k$. We also describe the limiting internal spatial
distribution of the $k$-clusters (suitably rescaled) at high intensity,
and show it is governed by an interesting energy functional defined on
configurations of $k$ points.

Formally, the
RGG based on a random sample
$\cX$ of points in $\R^d$ (with $d \in \N$) 
is the graph $G(\cX,r)$ with vertex set $\X$ and with an edge between each pair
of points distant at most $r$ apart, in the Euclidean metric, for a specified 
distance parameter $r >0$. Such graphs are important in a variety
of applications (see \cite{Pen03}), such as wireless communications 
(see \cite{BB09})
and
topological data analysis (TDA) (see \cite{BK18,BK22}).

The  PBM 
is given by a union of balls of equal radius, say $\frac12$, centred
a homogeneous Poisson process $\cH_\la$ of intensity $\la $ in $\R^d$.
The clusters of the PBM correspond to
those of 
the graph $G(\cH_\la,1)$,
which is often called the {\em Gilbert graph}.
The terminology `Poisson blob model'  is used in e.g.
\cite{Ale93, QT96}.
The PBM is a special case of the
{\em Poisson Boolean model} 
with balls of fixed radius. This
is a fundamental model of spatial clustering and continuum percolation in
stochastic geometry; see
 \cite{CSKM13}, \cite{HallBk}, \cite{LP18}, \cite{MR96},
 \cite{SW08}, \cite{Tor02}.

In this paper we consider, for fixed $k \in \N$, 
the {\em number of components of order $k$} (i.e., having
$k$ vertices) of the graph $G(\X,r)$,
here denoted $K_{k,r}(\X)$,
where $\X $ is either a random sample of $n$ points,
denoted $\X_n$, 
uniformly distributed over a compact set  
$A$ in $\R^d$
with a smooth  (in fact, $C^{1,1}$) boundary 
or the corresponding Poisson process, denoted $\Po_n$. In particular,
we investigate asymptotic properties of $K_{k,r}(\X_n)$
and $K_{k,r}(\Po_n)$ for large $n$ with
 $r=r(n)$ specified and decaying to zero according
to a certain limiting regime (see (\ref{e:supcri}) and
(\ref{e:sublog}) below).
In the special case $k=1$, $K_{1,r}(\X)$ is the number of isolated
vertices in $G(\X,r)$.

Some limiting regimes have been considered already.
In the {\em thermodynamic} limiting regime with $n r^d $ held constant,
$K_{k,r}(\X_n)$ 
is known to grow proportionately to $n$, with a strong law of large numbers
(LLN)
(see \cite[Theorem 3.15]{Pen03}) and  there are central limit theorems 
(CLTs) both for $K_{k,r}(\X_n)$  (see \cite[Theorem 3.14]{Pen03})
and for $K_{k,r}(\Po_n)$ (see \cite[Theorem 3.11]{Pen03}).
There is also a strong LLN for $K_{k,r}(\X_n)$ in 
the {\em sparse} regime with $nr^d \to 0$, subject to some further conditions
on $r(n)$; see \cite[Theorem 3.19]{Pen03}.
Within the sparse regime,
 when $n(nr^d)^{(k-1)} \to \infty$ but $n(nr^d)^{k}  \to 0$,
a CLT for $K_{k,r}(\X_n)$ can be derived from \cite[Theorem 3.5]{Pen03}
and the fact that the second condition implies that the number
of components of order larger than $k$ vanishes, in probability.
In fact the results in \cite{Pen03}
are stated for the number of components isomorphic to some specified
connected graph $\Gamma$,
but results for $K_{k,r}(\X_n)$ or
$K_{k,r}(\Po_n)$ can then be obtained by summing over all possible 
 $\Gamma$ with $k$ vertices. More recently, \cite{HO23}
 gives a large deviations principle for the $k$-component
 count in the sparse regime when also $n(nr^d)^{k-1} \to \infty$.

In the  {\em logarithmic regime} we take
$n \theta r^d = b \log n$ for some constant $b$, 
 where $\theta$ denotes the volume of the unit radius ball in $\RR^d$.
Provided $d \ge 2$ and
$b$ exceeds the critical value $b_c = (2-2/d) \vol(A)$
(where $\vol$ denotes Lebesgue measure),
it is known that with probability tending to 1 as $n \to \infty$,
$G(\X_n,r) $ is fully connected 
so that $K_{k,r}(\X_n)=0$  (see \cite[Theorem 13.7]{Pen03}).

The main limiting regime for $r=r(n)$ that we consider here is to make
the assumptions
\begin{align}\label{e:supcri}
	\lim_{n \to \infty} (nr^d) & = \infty; 
	\\
	\label{e:sublog}
	\limsup_{n \to \infty}((\theta n r^d)/ \log n) & < \min((2/d),1) \vol (A).
\end{align}
We call this the {\em intermediate} or 
{\em mildly dense} regime because the average
vertex degree is of order $\Theta(nr^d)$ and therefore grows
to infinity as $n$ becomes large, but only slowly in this regime.
Our intermediate regime includes the logarithmic regime for small
values of the constant $b$;
we discuss the significance of the upper bound in (\ref{e:sublog}) later on.

While asymptotics for the number of $k$-clusters
in continuum clustering models of this sort  are a classical object
of study
\cite{Pen03,
Hall86,HallBk,HZ77,Tor02,QT96}, there has also beeen renewed
interest 
more recently, e.g. \cite{HO23}, often motivated by
TDA, where topological properties
of a simplicial complex based on the sample of points are used to
try to understand those of the underlying space. The sum $\sum_k K_{k,r} (\X)$
is the {\em total number of components}  in $G (\X,r)$, and this
is of interest in TDA as one of the Betti numbers \cite{BK18,BK22,KM13}

In TDA the number of components (and other Betti numbers) are of
interest in the whole
range of values of $r$, not just in the sparse or thermodynamic limiting
regime. One can also imagine other situations where the dense limit is
of interest: for example the structure of impurities in a gas that
gets more and more compressed, or that of isolated communities
in a society that becomes more and more interconnected. 
However, most of the existing work on the limit theory
is concerned with the
sparse or thermodynamic limit; in the TDA context, this is noted in
the last paragraph of \cite[Section 2.4.1]{BK22}. 
It seems well worth building our understanding of the mildly dense limit 
as well.

In a companion paper \cite{PY25} we investigate the limiting behaviour of
the total number of components in the mildly dense limiting regime.
It turns out that this quantity is dominated by the first term in
the sum, namely $K_{1,r}(\X)$ (which is also the case in the sparse limit,
as observed long ago in \cite{Hall86}). The results in
\cite{PY25} rely on those in the present paper, in particular 
the variance asymptotics in Propositions 
\ref{p:var_k} and \ref{p:var_iso_bin} and
the Poisson approximation results in Theorem \ref{t:int_k}. 
While it is the special case $k=1$ of our results here that are most
relevant to \cite{PY25}, one can imagine some algorithm whereby all singletons
are removed from the sample, and one is interested in the number of remaining
clusters. In this case we would expect the 2-clusters to dominate, and
if these too were removed, the 3-clusters would dominate, and so on. So it
seems natural to investigate the behaviour of $k$-clusters for general $k$.

\subsection{Summary of results}

Regarding the Poisson blob model, the number of components of order $k$
in $G(\cH_\la,1)$ will be infinite. Instead, we consider the
probability $p_k(\la)$ that an inserted point at the origin
$o$ lies in a component of order $k$ of $G(\cH_{\la} \cup \{o\},1)$.
This can be interpreted, loosely speaking, as the proportion
of vertices in $G(\cH_\la,1)$ that lie in components of order $k$.

Our first result, Theorem \ref{t:alpha},
determines the asymptotic behaviour of
$p_k(\la)$ for large $\la$ and fixed $k \in \N$.
It says that $p_k(\lambda) \sim \alpha_k\lambda^{(1-k)(d-1)}e^{-\theta
\lambda}$ as $\lambda \to \infty$,
where $\alpha_1 :=1$ and for $k \in \N$, $\alpha_{k+1}
: = (1/k!) \int_{(\R^d)^k} e^{-g(\bz)} d\bz$, where
$$
g(\bz) := \int_{\cup_{i=1}^k B (\frac12 z_i,\|z_i\|/2)} \|x\|^{1-d} d x,
~~~~ \bz = (z_1,\ldots,z_k) \in (\R^d)^k.
$$
Here and elsewhere, for $x \in \R^d$, $r>0$ we write  $B(x,r) $ or $ B_r(x)$
for $\{y \in \R^d: \|y-x\|\leq r\}$,
and  $\|\cdot\|$ denotes the Euclidean norm on $\R^d$.
This result is both of interest in its
own right, and important for understanding the asymptotics
for number of components of order $k$ in RGGs in the mildly dense regime.

In Theorem \ref{t:ptprlim}, we shall provide
%
%
a result on the large-$\lambda$
limiting distribution of the (rescaled) points of
 the 
component containing the origin,
given that it is of order $k+1$. Namely, their joint density 
is given by $e^{-g(\bz)}$, normalized to a probability density function. 
We have not seen the `energy functional' $g(\cdot)$ in the literature before.



We now describe our main results on finite RGGs.  Let
$A \subset \R^d$ be compact with smooth boundary (in the $C^{1,1}$ sense
that we define later), and let
 $X_1,X_2,\ldots$  be independent identically distributed random $d$-vectors,
 uniformly distributed over $A$.
 For $n\in\NN$ set
 $\cX_n:=\{X_1,\ldots,X_n\}$,  which is a binomial point process
 (see e.g. \cite{LP18}).
 Also, for $n \in (0,\infty)$  
 let $Z_n$ be a Poisson random variable with mean $n$,
 independent of $(X_1,X_2,\ldots)$, and set
 $\Po_n:=\cX_{Z_n}$. Then
 (see \cite[Proposition 3.5]{LP18})
 $\Po_n$
 is a Poisson point process in $A$ with intensity measure $(n/\vol(A))dx$
 (in this case  $n$ does not need to be an integer). 
 
 Our results are concerned with the variables $S_{n,k} := K_{k,r}(\cX_n)$ and
$S'_{n,k} :=K_{k,r}(\Po_n$), with $r=r(n)$ given, satisfying
(\ref{e:supcri}) and (\ref{e:sublog})  unless stated otherwise.
For now  we present them as asymptotic results as $n \to \infty$
with $k$ fixed, but the precise statements of these results
later on come with bounds on the rates of convergence.

As first-order results we shall give the asymptotic
behaviour of $\E[S_{n,k}]$ and $\E[S'_{n,k}]$. 
Setting $I_{n,k} := \E[S'_{n,k}]$, we show in Theorem \ref{t:Ilim} that
(with $\alpha_k$ given above)
\begin{align}
	I_{n,k} \sim 
	k^{-1} \alpha_k n(nr^d/\vol(A))^{(1-k)(d-1)}
	\exp(- (\theta / \vol(A)) n r^d )
	\label{e:Inasymp}
\end{align}
as $n \to \infty$.
We show in Proposition
	\ref{p:Snmeanbin} that also
$\E[S_{n,k}] \sim I_{n,k}$ as $n \to \infty$.

Turning to second-order results, we shall 
show in Propositions \ref{p:var_k} and \ref{p:var_iso_bin}
that 
\begin{align}
	\Var[S_{n,k}] \sim 
	\Var[S'_{n,k}] \sim I_{n,k}  
	~~~~~~~~~~~~~~~
	{\rm as} ~n \to \infty.
	\label{e:Varlim}
\end{align}


Note that (\ref{e:Inasymp}), (\ref{e:supcri}) and (\ref{e:sublog})
together imply that $I_{n,k} \to \infty$ as
$n \to \infty$.
It is immediate from the asymptotics already described, and Chebyshev's
inequality, that we have the
weak LLNs $(S_{n,k}/I_{n,k}) \toP 1$ and $(S'_{n,k}/I_{n,k})
\toP 1$ as $n \to \infty$,
and hence 
using (\ref{e:Inasymp}) we have 
	$S_{n,k} /(k^{-1} \alpha_k n(nr^d/\vol A)^{(1-k)(d-1)}e^{-f_0 \theta n r^d}) \toP 1$
	and likewise for $S'_{n,k}$, where $c$ is as before.  In Theorem
	\ref{t:non-uniform},
	under the extra condition $\limsup((n\theta r^d)/\log n) < \vol(A)/2$,
	we improve the weak LLNs
	to strong LLNs, i.e. we
	prove almost sure convergence
	\begin{align*}
		S_{n,k} /(k^{-1} \alpha_k n(nr^d/\vol(A))^{(1-k)(d-1)}e^{-f_0 \theta n r^d}) \toas 1
		~~~~{\rm as}~ n \to \infty,
	\end{align*}
	and likewise for $S'_{n,k}$.
	We do this by giving
	concentration of measure results for $S_{n,k}$ and for $S'_{n,k}$.

Turning to CLTs,
in Theorem \ref{t:int_k}
we show that as $n \to \infty$, 
\begin{align}
I_{n,k}^{-1/2}(S'_{n,k} - I_{n,k}) \toD N(0,1);~~~~~
I_{n,k}^{-1/2}(S_{n,k} - \E[S_{n,k}]) \toD N(0,1).
	\label{e:CLTgen}
\end{align}
We shall give two approaches to proving (\ref{e:CLTgen}); one via
approximating $S_{n,k}$ 
by a Poisson distribution with
mean $I_{n,k}$ (of interest in itself), and the other via more direct normal
approximation. Both methods provide bounds on rates of convergence
in various metrics on probability distributions.

 Finally, in Section \ref{s:sparse}  we shall present some new results in the {\em sparse} limiting
regime, where instead of (\ref{e:supcri}) we assume $nr^d \to 0$ and
$n(nr^d)^{k-1} \to \infty$ as $n \to \infty$. In this case,
instead of (\ref{e:Inasymp}) we have $I_{n,k} \sim c n(nr^d)^{k-1}$ for
appropriate $c$.   Again in
this regime we can show that $\E[S_{n,k} ] \sim I_{n,k}$, and that
(\ref{e:Varlim}) and (\ref{e:CLTgen}) still hold, by similar arguments
to those we use in the mildly dense regime.

It is natural to ask whether our results on $S_{n,k}$ and $S'_{n,k}$ 
generalize to non-uniform distributions.  To answer this, we shall present
these  results in greater generality.  We shall assume $X_1,X_2,\ldots$
are independent random $d$-vectors having a common probability distribution 
$\nu$ with density $f$, supported by $A$ (so that $ \Pr[X_i \in dx] = \nu(dx)  
= f(x)dx $ for $x \in A, i \in \N$).
We then define $\X_n$ and $\Po_n$ as already described.

For all of our results on the mildly dense limiting regime,
we shall assume that $f$ is continuous on $A$,
and that $f_0 >0$, where we set 
$$
f_0:= \inf_{x \in A} f(x); ~~~ f_1 := \inf_{x \in \partial A} f(x); ~~~
		\fmax:= \sup_{x \in A} f(x),
$$
and $\partial A$ denotes the boundary of $A$.  By compactness of $A$
and continuity of $f$, we have $\fmax < \infty$; moreover $f_1 \geq f_0$.
 In the special case where $f$ is constant on $A$,
 we have $f_0=f_1=\fmax = 1/\vol(A)$; we call this the {\em uniform case}.

 In the general (non-uniform) case, instead of (\ref{e:sublog}) we assume
 \begin{align}
	\limsup_{n \to \infty}((\theta nr^d)/ \log n)  <
	 \frac{1}{\max(f_0,d (f_0 - f_1/2)) },
	 \label{e:sublog2}
 \end{align}
 which reduces  to (\ref{e:sublog}) in the uniform case.
 This condition may be understood as follows.  Suppose
 $ \theta n r^d = b \log n$ for some constant $b$. 
 A `back-of-the envelope' calculation suggests
 that the mean number of components of order $k$
 (which we shall call {\em $k$-components} or {\em $k$-clusters} for short)
 in the interior of $A$
 is roughly of the order $n (nr^d)^{k-1} 
 e^{-f_0 \theta n r^d} \approx n^{1-b f_0}$
 (up to logarithmic factors), which tends to infinity if $b <1/f_0$.
 Similarly
 the mean number of $k$-clusters 
 within distance $kr$ of the boundary of $A$ is of order
 $nr (nr^d)^{k-1}e^{-f_1 \theta n r^d/2} \approx n^{1-(1/d)- (b f_1/2)}$.
 If $b$ is less than the right hand side
 of (\ref{e:sublog2}), then $ bf_0 < (1/d)+ bf_1/2 $,
 so there are more $k$-clusters
 in the interior of $A$ than near the boundary, 
 suggesting we can ignore boundary effects when
 analysing the asymptotic behaviour of $S_{n,k}$ and $S'_{n,k}$.
 If $b$ were larger than this, we would have to take boundary effects
 into account more carefully, for example to determine
 the analogue to (\ref{e:Inasymp}) in the uniform case,
 and we leave this as possible
 future research.

 All of the results for $S_{n,k}$ and $S'_{n,k}$  described earlier 
 for the uniform case remain valid in the general (non-uniform) case under
 assumptions (\ref{e:supcri}) and (\ref{e:sublog2}), {\em except} for the
 asymptotic (\ref{e:Inasymp}).  In general, $I_{n,k} = \E[S'_{n,k}]$ is
 equal to a multiple integral, given  at (\ref{e:exp_Snk}).
In the non-uniform case we are not able to
describe the limiting behaviour of $I_{n,k}$ so explicitly, but
we do show that $I_{n,k} \to \infty$ and
$(nr^d)^{-1} \log( I_{n,k}/n) \to - \theta f_0$
as $n \to \infty$; see Lemma \ref{l:Ilower} and Theorem
\ref{t:Inonunif}.

\begin{remark}
	\begin{enumerate}
	\item
		As alluded to earlier, in the thermodynamic or
			sparse limiting regime, it is of interest
			to further distinguish between different
			components of order $k$, according to which
			of the possible graphs $\Gamma$ they are
			isomorphic to. In the mildly dense limiting regime
			considered here, this
			is perhaps of less interest, because of the {\em compression}
			phenomenon already noted in \cite{Ale93}; for
			fixed $k$ and large $n$, nearly all
			of the components  of order $k$ are likely to
			be compressed into balls of diameter less than
			$r$ (in fact, less than $\delta r$ for any 
			fixed $\delta>0$), and therefore
			fully connected. In other words, nearly all
			of the components
			of order $k$ are likely to be cliques.

		\item
			 It seems plausible
			that one could derive a result on the
			limiting within-cluster spatial distribution of points in $k$-clusters
			of the RGG in the mildly dense limit,
			suitably rescaled,
			using  Theorem \ref{t:ptprlim}, but we do not do so here.

In \cite{QT96} the mean volume of $k$-clusters of the PBM at finite intensities
is investigated, in particular providing an expansion in $\lambda$ for 
$\lambda$ small. We would expect Theorem 
\ref{t:ptprlim} to be useful for understanding the limiting behaviour
			of the mean volume of $k$-clusters 
for $\lambda $ large.

\item
To prove our Gaussian approximation  results, we shall use the
			local dependence methods of Chen and Shao
			\cite{CS} to deal
			with $S'_{n,k}$ and those of Chatterjee
			\cite{Cha08} to deal with $S_{n,k}$.
For the Poisson approximation, we shall use the coupling bounds from 
e.g. Lindvall \cite{Lin92} for $S_{n,k}$ and from Penrose
			\cite{Pen18} for $S'_{n,k}$.
		All of these approximation bounds 
  are obtained by a blend of Stein's method and stochastic analysis of
	general Poisson or binomial point process functionals.  For a
	systematic account of the subject, we refer the reader
to \cite{Ste86, BHJ92,CGS11,NP12,LP18}, and  to \cite{APY21}
for some recent developments about the methodology. 
\item
Our results hold for all $d \geq 2$, and also for $d=1$ with the sole
exception of our result on Poisson approximation for $S_{n,k}$
(Proposition \ref{p:PoApproxBi}); we conjecture that even this result could
be proved by other means for $d=1$.  Many of our results seem to be new
even for the special case with $k=1$ 
(i.e., the isolated vertex count) and/or $d=1$.
%
	\end{enumerate}
\end{remark}

\subsection{Notation}
We use the following notation for asymptotics.
Given $g: (0,\infty)\to \R$,
and $h: (0,\infty)\to (0,\infty)$,
we write $g(x) = O(h(x))$ if $\limsup |g(x)|/h(x)<\infty$, and 
$g(x)=o(h(x))$ if $\limsup |g(x)|/h(x)=0$, and
$g(x) = \Omega(h(x))$ if $\liminf |g(x)|/h(x) > 0$, and
$g(x) = \Theta(h(x))$ if both
$g(x) =O(h(x))$ and $g(x) = \Omega(h(x))$,
and $g(x) \sim h(x)$ if $g(x)/h(x) $ tends to $1$.
Here, the limit is taken either as $x\to 0$ or $x\to \infty$, to
be specified in each appearance.

We denote the origin in $\R^d$
by $o$, and write just $B_r$ for $B_r(o)$.
Also for finite
$\X \subset \R^d$ and $r>0$, we define
the set $B_r(\X):= \cup_{x \in \X} B_r(x)$, and the
function $h_r(\X)$ to be the indicator of the event that
$G(\X,r)$ is connected. Also for $m \in \N$ and 
$\bx = (x_1,\ldots,x_m) \in (\R^d)^m$ we define 
$B_r(\bx)$ and $h_r(\bx)$ similarly, identifying
$\bx$ with $\{x_1,\ldots,x_m\}$; that is
\bea
B_r(\bx) := \cup_{i=1}^m B_r(x_i); ~~~
h_r(\bx) = \1\{G(\{x_1,\ldots,x_m\},r) ~{\rm is~ connected}\}.
\label{e:Brhrdef}
\eea
Also we define $h^*_r(\bx)$ to be the product of $h_r(\bx)$
and the indicator of the event that $x_1$ precedes all of
$x_2,\ldots,x_m$ in the lexicographic ordering (that is,
$x_1$ is the left-most point of $\{x_1,\ldots,x_m\}$).

For $A \subset \R^d$,
we set $A^{(r)}:= \{x \in A: B_r(x) \subset A\}$,
and $A^o:= \cup_{r >0} A^{(r)}$, the interior of $A$.
If $\cX \subset \R^d$,
and $a >0$
let $a \cX:= \{ax:x \in \cX\}$. 
If also $\cY \subset \R^d$ and $\cX$ and $\cY$ are locally finite
and non-empty, let $\dist(\cX,\cY) := \min_{x \in \cX,y\in \cY} \|y-x\|$.

For $n \in \N$, set $[n]:= \{1,\ldots,n\}$.
If $\ell,m \in \N$ and $\bx = (x_1,\ldots,x_\ell) \in (\R^d)^\ell$,
$\by = (y_1,\ldots,y_m) \in (\R^d)^m$, then we set
$\dist(\bx,\by) := \min_{i \in [\ell],j \in [m]} \|x_i -y_j\|$.

Given a random vector $\bX$ and an event $\mathscr E$ on the same
probability space, let  $\LL(\bX)$ denote the
probability distribution (also called the law)
of $\bX$, and $\LL(\bX|\mathscr E)$ the
conditional probability distribution of $\bX$ given that
event $\mathscr E$ occurs.

Given $r >0$, given locally finite $\X \subset \R^d$
and given $x \in \X$, let $\cC_r(x,\X)$ 
denote the vertex set of the component of the
graph $G(\X,r)$ containing $x$, and
$|\cC_r(x,\X)|$ the number of elements of this set. 

Given $m \in \N$ and $p \in [0,1]$,
we write $\Bin(m,p)$ for a  Binomial
		 random variable with parameters $m,p$,
\section{High-intensity Boolean model}
\label{s:Boolean}


Given $\lambda >0$,
 let $\cH_\la$ be a homogeneous Poisson point process
in $\R^d$ of intensity $\la$ (viewed as a random set of points),
 and let $\cH_{\la,0} := \cH_\la \cup \{o\}$.
Then for $r>0$ the graph $G(\cH_\la,r)$ is the intersection graph of
 a PBM, i.e. a Boolean model of continuum percolation 
 with balls of fixed radius $r/2$. See \cite{LP18,MR96} for background
 information on models of this type.

Let $\cC(\la):=\cC_1(o,\cH_{\la,0})$.
For $k \in \N$
let $p_k(\la):= \Pr[|\cC(\la)|=k]$. 
This section is concerned with the large-$\lambda$ asymptotics of 
$p_k(\la)$ for fixed $k$,
and the asymptotic conditional 
distribution of the point process 
$\cC(\la)$ given $|\cC(\la)|=k$. 
These are relevant to the RGG model described in
the Section \ref{secintro}, as we shall show later on.

In  \cite[Theorem 2.2]{Ale93} Alexander proved (among other things) that as
$\la \to \infty$  we have
\begin{align}
p_{k+1}(\la) = \Theta \left( \la^{-k(d-1)} e^{-\theta \la}\right). 
\label{e:Alex}
\end{align}
In other words, $\la^{k(d-1)} e^{\theta \la} p_{k+1}(\la)$
remains bounded away from $0 $ and $\infty$ as $\la \to \infty$.
This suggests that this quantity might tend to a positive finite
limit as $\la \to \infty$, 
but this is not proved in \cite{Ale93}. Our next result shows
that this is indeed the case, and provides a rate of convergence. To
describe the limit, we define
for each $\bz = (z_1,\ldots,z_k) \in (\R^d)^k$ the set
\bea
D(\bz) := \cup_{i=1}^k B
((1/2)z_i, \|z_i\|/2),
\label{e:Dz}
\eea
that is, the union of balls $B_1^*,\ldots,B_k^*$,
where for each $i$ the ball $B_i^*$ has opposite poles at $o$ and $z_i$
(Figure 1 shows an example with $d=2$).
Define the function $g:(\R^d)^k \to [0,\infty)$ by
\bea
g(\bz) = \int_{D(\bz)} \| x\|^{1-d} dx,  ~~~~ \bz \in
(\R^d)^k.
\label{e:gz}
\eea
When $d=2$, this can be interpreted as the {\em gravitational energy}
of a flat object with the shape of $D(\bz)$ and mass equal
to the area of $D(\bz)$, and with its mass evenly spread,
with respect to a large point mass at the origin;
that is, the energy required to remove the object from the
gravitational field of the point mass.
We are not aware of any physical interpretation of $g(\cdot)$
in other dimensions, but nevertheless refer to $g(\bz)$
as the {\em quasi-gravitational energy} of the set $D(\bz)$.
Set $\alpha_1 := \alpha_1(d):=1$, and for $k \in \N$
define the constant
		\bea
		\alpha_{k+1}:= \alpha_{k+1}(d) : = (1/k!) \int_{(\R^d)^k}
\exp(-g(\bz))d \bz.
\label{e:alpha}
\eea

%
\begin{theorem}
	\label{t:alpha}
	Let $k \in \N \cup \{0\} $. Then as $\la \to \infty$, 
\begin{align}
	\la^{k(d-1)} e^{\theta \la} p_{k+1}(\la) =
	\alpha_{k+1}+ O(\la^{-1}).
\label{e:Alim}
\end{align}
\end{theorem}
The existence of the limit in \eqref{e:Alim} is crucial for our other
results later on, in which we shall  establish
the existence of limiting constants in the limit theory for the
mean and variance of the number of $k$-components of the finite
RGG in the mildly dense
limit.

Let $k \in \N$.  Another part of \cite[Theorem 2.2]{Ale93} says, in effect, that
conditionally on $|\cC(\la)|=k+1$, the random variable
$\lambda \diam (\cC(\la))$ is bounded away from zero and infinity
in probability (here $\diam$ is Euclidean diameter);
in other words, $\diam(\cC(\la))= \Theta(\la^{-1})$ in probability
as $\la \to \infty$, given $|\cC(\la)|=k$.
This phenomenon is known as
{\em compression} \cite{Ale93,MR96}: the density
of points in the cluster $\cC(\la)$ is of the order of $\la^d$,
compared to an ambient density of $\la$.

We  give a stronger result here, namely convergence
in distribution of the point process $\la \cC(\la)$.
To describe the limit, let $(U_1,\ldots, U_k)$ be a random vector
in $(\R^d)^k$ with  joint density with respect to $(dk)$-dimensional
Lebesgue measure, given by $ (k! \alpha_{k+1})^{-1} \exp(-g(\cdot))$,
i.e. with 
\begin{align}
\Pr[ (U_1,\ldots,U_k) \in d\bz ] = (k! \alpha_{k+1})^{-1}
\exp(-g(\bz)) d \bz, ~~~~~
\bz
\in (\R^d)^k,
	\label{e:limPP}
\end{align}
with the energy function $g$  given at (\ref{e:gz}), and $\alpha_{k+1}$
given at (\ref{e:alpha}).  Informally, our result says that the point process
$\la \cC(\la)\setminus \{o\}$ converges in distribution
to the point process $\{U_1,\ldots,U_k\}$. 

To make
this precise while avoiding technicalities on point process
convergence, we list the points of our point processes in increasing 
lexicographic order to get random vectors in $(\R^d)^k$
(alternatively we could use distance from the origin for the ordering). 
Let $Y_{1,\la},\ldots,Y_{k,\la}$ be the points of $\cC(\la) \setminus \{o\}$
taken in increasing lexicographic order, and 
let $V_1,\ldots,V_k $ be the points of $\{U_1,\ldots,U_k\}$ taken
in increasing lexicographic order.

\begin{theorem}
	\label{t:ptprlim}
The conditional distribution of the random  vector  
$
	 ( \la Y_{1,\la},\ldots,\la Y_{k,\la}),
$
	given  $|\cC(\la)|=k+1$, converges weakly,
	as $\la \to \infty$, to  the distribution of
	$(V_1,\ldots,V_k)$.
\end{theorem}

\begin{remark}
	\begin{enumerate}
		\item
	The result (\ref{e:Alex}) (but not (\ref{e:Alim}))
	also appears in \cite[Section 5.3]{MR96}.
The results in \cite{Ale93,MR96} are also presented for components
of order $k+1$ in $G(\cH_{\la,0},s)$ for general fixed $s>0$ (rather than
just for $s=1$).  It is easy to generalize (\ref{e:Alim}) to this setting,
since by Poisson scaling (see e.g. \cite{MR96}, or the Mapping Theorem
			in \cite{LP18})
$$
\Pr[|\cC_s(o,\cH_{\la,0})|=k] =\Pr [|\cC_1(o,s^{-1} \cH_{\la,0})|=k]
			= p_k(\la s^d).
$$

\item When $d=1$ we can compute $\alpha_{k+1}$ exactly.  In this case we have
$g(z) = \max_{i \in [k]} z_i^+ + \max_{i \in [k]} z_i^-$,
and considering separately the case where $z_1, \ldots, z_k$ all
have the same sign and the case where they do not, we obtain that
\begin{align*}
	k!	\alpha_{k+1} & = 2k \int_0^\infty z^{k-1} e^{-z} dz
	+  k (k-1)\int_0^\infty du \int_0^\infty dv (u+v)^{k-2} e^{-(u+v)} 
	\\
& = 2k! + k(k-1) \int_0^\infty w^{k-2} e^{-w} w dw \\
		& = 2k! + k!  (k-1),
\end{align*}
		so $\alpha_{k+1}(1) = k+1$. 
\item
	We can also compute $\alpha_2(d)$ more explicitly for general $d$.
	Indeed, when $k=1$ it can be seen directly from
			(\ref{e:gdef}) below that $g(z)= \lim_{r \downarrow 0}
			g_r(z) = \theta_{d-1}\|z\|$
	for all $z \in \R^d$, where $\theta_{d-1}$ denotes the volume of a unit
	radius ball in $d-1$ dimensions.  Then by (\ref{e:alpha}) we have
\begin{align*}
	\alpha_2(d) & = \int_{\R^d} \exp(- \theta_{d-1} \|z\|) dz
	= \theta_{d-1}^{-d} \int_{\R^d} e^{- \|u\|} du
	= \theta_{d-1}^{-d} d \theta \int_0^\infty e^{-r}r^{d-1} dr
\\
&	= d! \theta/ \theta_{d-1}^{d}. 
\end{align*}
\item It may be possible to improve the right hand side of (\ref{e:Alim}) to
an infinite power series in $\la^{-1}$ and to compute the first few
coefficients.  This series expansion would provide a high-intensity
complement to the low-intensity expansions in $\la$ which have been considered
previously in the literature dating back to \cite{HZ77}; see \cite{QT96},
\cite[p.242]{Tor02} and references therein.
	\end{enumerate}
\end{remark}
For the proof of Theorem \ref{t:alpha}, we introduce the following notation.
Given $\bx = (x_1,\ldots,x_k) \in (\R^d)^k$, recalling 
$B_r(\bx):= \cup_{i=1}^k B_r(x_i)$, we define 
\bea
V(\bx):= \vol(B_1(\bx));
~~~~~~
V'(\bx)  :=\vol( B_1(\bx) \setminus B_1(o));
~~~~~~
\lvvvert \bx \rvvvert := \max_{1 \leq i \leq k} \|x_i\|.
\label{e:Vpr}
\eea
The proof of Theorem \ref{t:alpha} will use the following 
geometrical lemma, which begins to show the 
relevance of $g(\bz)$, defined above at (\ref{e:gz}) 
as the quasi-gravitational energy
of $D(\bz)$.

\begin{lemma}
	\label{l:QGE}
	For $\bz \in (\R^d)^k$, and $r>0$, set $g_r(\bz):= r^{-1}
	V'(r\bz)$. There exists a constant $c=c(d)\in (0,\infty)$
	such that for all $r \in (0,1]$ and 
	$\bz \in (\R^d)^k$ with $r \lvvvert \bz \rvvvert \leq 1$,
\bea
		 |g_r(\bz) -  g(\bz)| \leq c r \lvvvert \bz \rvvvert^2 . 
\label{e:gdef}
\eea
\end{lemma}

\begin{proof}
When $d=1$, it is easy to see directly that
	$g_r(\bz)= \max_i z_i^+ + \max_i z_i^- = g(\bz)$
	whenever $r \lvvvert \bz \rvvvert \leq 1$, and hence that
(\ref{e:gdef}) holds in this case.
Therefore we can and do assume from now on that $d \geq 2$.

Let $r \in (0,1]$ and $ \bz= (z_1,\ldots,z_k) \in (\R^d)^k$ with
$r \lvvvert \bz \rvvvert \leq 1$.  
	Using polar coordinates, write
$$
g_r(\bz) = r^{-1} \int_{\partial B_1} \int_0^\infty  \1 \{(1+s)x \in
\cup_{i=1}^k B_1(rz_i) \} (1+s)^{d-1} ds \sigma(dx),
$$
where the outer integral is a surface integral with the surface measure denoted
$\sigma$. 

Given $x \in \partial B_1$, $y \in B_1$, let 
$ s(x,y)  := (\max\{s: (1+s)x  \in B_1(y) \})^+$. Then 
\begin{align}
	g_r(\bz) & = r^{-1} \int_{\partial B_1} \int_0^{\max_i s(x,rz_i)}
	(1+s)^{d-1} ds
	\sigma(dx)
	\nonumber \\
	\label{e:polar}
	& = \int_{\partial B_1} (rd)^{-1} ( (1+ \max_i s(x,rz_i))^d -1) 
	\sigma(dx). 
\end{align}
By the triangle inequality, $s(x,rz_i) \leq r \|z_i\|$, so
	the integrand in (\ref{e:polar}) is bounded by
	a constant independent of $x$. 
	We shall prove that the limit $\lim_{r \downarrow 0}r^{-1} s(x,rz_i)$
exists, for each $i \in [k]$ and $x \in \partial S$, with
	a rate of convergence. 
 \begin{figure}[!h]
\label{fig0}
\center
\includegraphics[width=8cm]{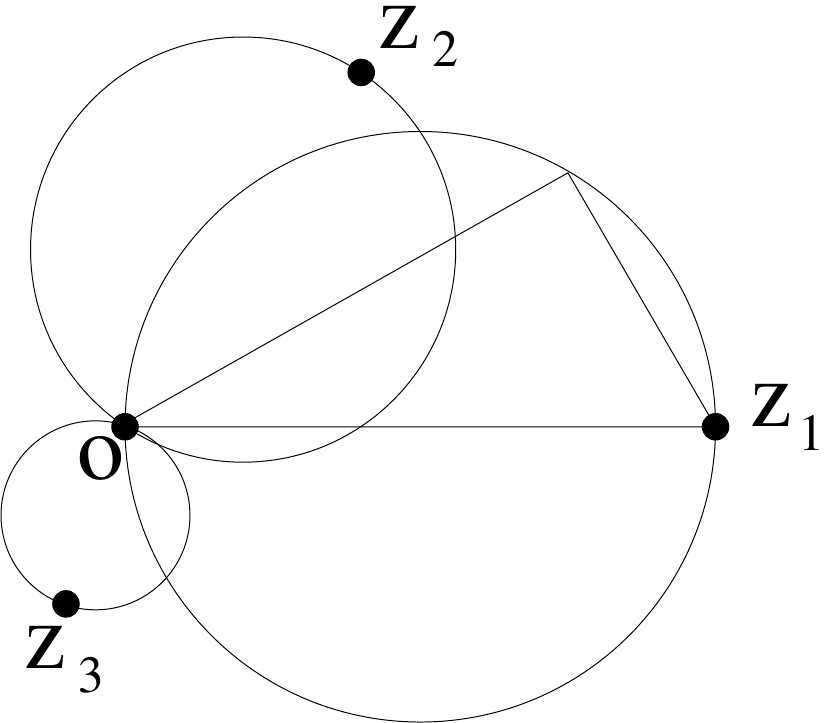}
	 \caption{The union of the disks is the set $D((z_1,z_2,z_3))$.}
\end{figure}

 \begin{figure}[!h]
\label{fig1}
\center
\includegraphics[width=8cm]{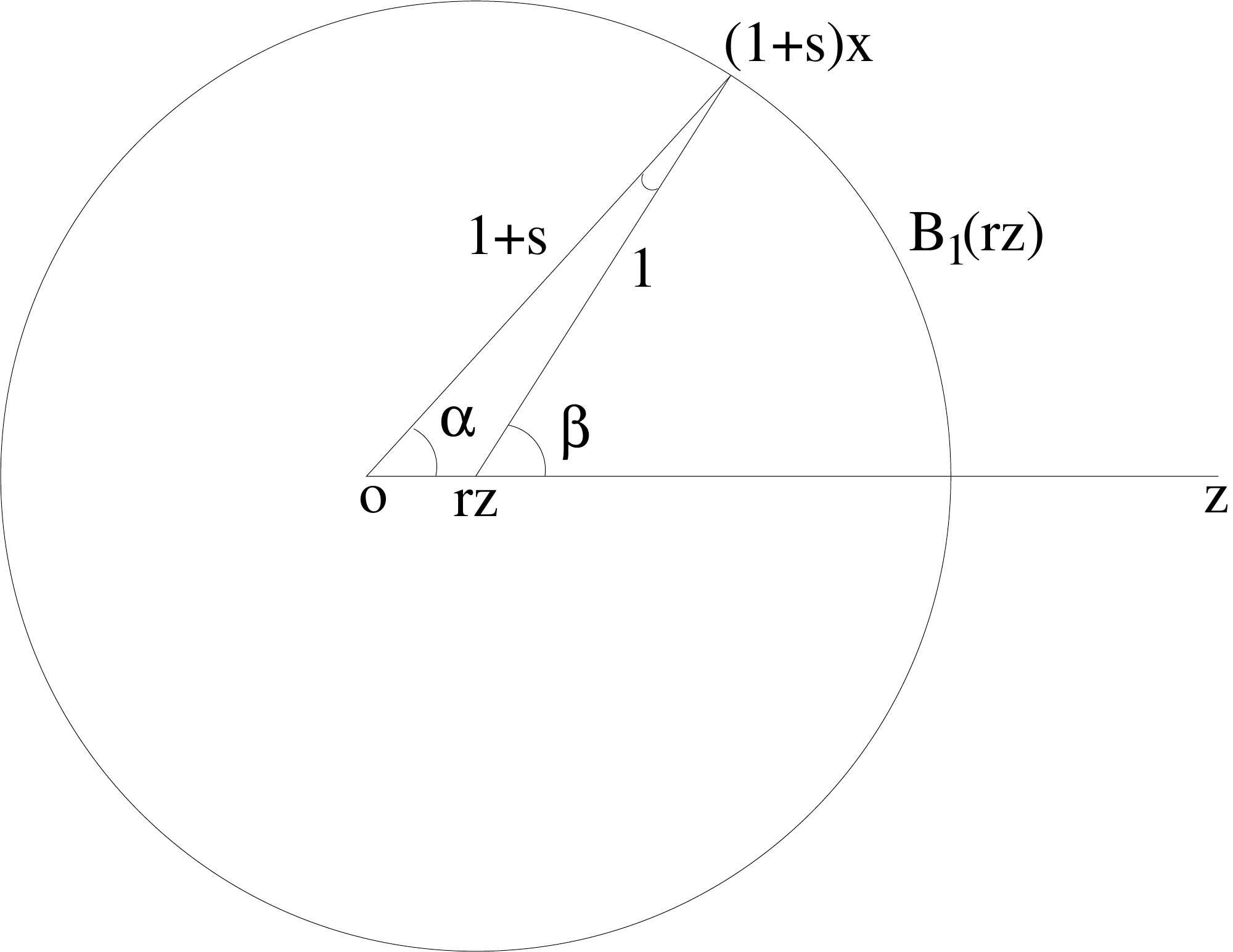}
	 \caption{Illustration of the proof of Lemma \ref{l:QGE}.}
\end{figure}

The reader is invited to refer to Figure 2 now.
	Let $i \in [k]$ and set $z=z_i$. If $z=o$ then
$s(x,rz)=0$.  Given $r \in (0,1]$,   $z \in B_{1/r} \setminus \{o\}$
and $x \in \partial B_1$, let $\alpha = \alpha(x,z)$ be the angle between 
$x$ and $z$, i.e the angle of $(x ,o, z)$, and assume $\alpha < \pi/2$
(otherwise, it is not hard to see that $s(x,rz) =0$).  Set $s= s(x,rz)$.
Let $\beta = \beta(r,x,z) $ be the angle of $((1+s)x, rz,z)$,
i.e. the angle between $(1+s)x -r z$ and $z$.  Then  $\|(1+s) x - rz\| =1$, 
and by the cosine rule,
$$
(1+s)^2 = \|rz +((1+s) x - rz )\| =
(r \|z\|)^2 + 1 + 2 r \|z\| \cos \beta. 
$$
Also $\beta > \alpha$, 
and $\beta-\alpha$ is the most acute angle shown in Figure 2;
in other words,
$\beta - \alpha = \arcsin
(r\|z\| \sin \alpha)$.

If $t = \sin \omega$ with $\omega \in [0,\pi/2]$ then
$t \geq (2/\pi) \omega$, so that $\arcsin(t) = \omega \leq (\pi/2) t$.
Hence
for all $r\in (0,1]$, $z \in B_{1/r} \setminus \{o\}$
and all $x \in \partial B_1$ such that
$\alpha(x,z) \in [0,\pi/2)$, we have $0< \beta(r,x,z)
-\alpha(x,z) \leq (\pi/2) r \|z\|$. 

There is a constant $c_1$ such that for all $t \geq -1$
we have $|(1+t)^{1/2}- 1 -t/2| \leq c_1 t^2$. Hence
there is another constant $c_2$ such that
for all $r \in (0,1]$, $z \in B_{1/r} \setminus \{o\}$
and $x \in \partial B_1$ with $\alpha(x,z) \in (0,\pi/2)$ we have
\begin{align*}
	|s(x,rz)-r\|z\| \cos \alpha(x,z)|
	= |(1 + 2 r\|z\| \cos \beta(r,x,z) + r^2\|z\|^2 )^{1/2} - 
	1 - r \|z\|\cos \alpha(x,z)|
	\\
	\leq |r\|z\|(\cos \beta(r,x,z) - \cos \alpha(x,z)) | + 
	\frac{r^2}{2}\|z\|^2
	+ c_1(r\|z\|\cos \beta(r,x,z) + r^2\|z\|^2)^2
	\\
	\leq c_2 r^2 \|z\|^2. 
	~~~~~~~~~~~~~~~~~~~~~~~~~~~~~~~~~~~~~~~~~
	~~~~~~~~~~~~~~~~~~~~~~~~~~~~~~~~~~~~~~~~~
	~~~~~~
\end{align*}
There is a further constant $c_3$ such that for $t,u \in [0,1]$ we have 
$ |(1+u)^d - 1- dt|\leq c_3 (|u|^2 + |u-t|).  $
Therefore
$$
((1+ \max_i s(x,rz_i))^d -1 - rd \max_i \{\|z_i\| (\cos \alpha(x,z_i))^+ \}
\leq c_3 (\max_i s(x,rz_i)^2 + c_2 r^2 \lvvvert \bz \rvvvert^2).
$$
Hence by (\ref{e:polar}), and the inequality
$s(x,rz_i) \leq r\|z_i\|$, there is a constant $c_4$ such that
\begin{align}
	\left| g_r(\bz) -  
	\int_{\partial B_1}  
	 \max_{i \in [k]} \{ \|z_i\| (\cos \alpha(x,z_i))^+ \} \sigma(dx)
	 \right| \leq c_4  r \lvvvert \bz \rvvvert^2.
	 \label{e:gz1}
\end{align}
Using 
Thales' theorem, one can show that
 $\max_i \|z_i\| (\cos \alpha(x,z_i))^+$ is the distance from
the origin to the furthest point from the origin in the direction of $x$
lying in the set $D(\bz)$, defined at (\ref{e:Dz}) (this
is illustrated in Figure 1; the triangle shown there is right angled.).
Therefore using polar coordinates again, we can verify that
$g(\bz)$, defined at (\ref{e:gz}), equals the integral in
(\ref{e:gz1}).
\end{proof}

	\begin{proof}[Proof of Theorem \ref{t:alpha}]
		Since $\Pr[|\cC(\la)|=1] = e^{-\lambda \theta}$,
		and we set $\alpha_1=1$,
	(\ref{e:Alim}) is immediate for $k=0$.
Therefore we can and do assume from now on that $k \geq 1$.

	Recall the definitions  of $h_r(\cdot)$ and $B_r(\cdot)$
		at (\ref{e:Brhrdef}), and $V(\cdot)$ at (\ref{e:Vpr}).
It is known (see e.g. \cite{Pen91}) that
\begin{align}
p_{k+1}(\la) = \frac{\la^{k}}{k!} \int_{\R^d} \cdots \int_{\R^d}
	h_1((o,x_1,\ldots,x_k)) \exp(- \la V((o,x_1,\ldots,x_k))) 
	dx_k \cdots dx_1.
	\label{e:pk}
\end{align}
This formula is a consequence of the multivariate Mecke formula
		(see e.g. \cite{LP18}); a similar formula to
		(\ref{e:pk}) is derived later on  at 
		(\ref{e:exp_Snk}).

		Since $V((o,x_1,\ldots,x_k)) = V'((x_1,\ldots,x_k)) + \theta$ by
		(\ref{e:Vpr}), we obtain from (\ref{e:pk})  that
		\begin{align}
			p_{k+1}(\lambda) & = \frac{\lambda^k e^{-\la \theta}}{k!} \int_{(\R^d)^k } h_1(o,\bx)
\exp(- \la V'(\bx)) d\bx
\nonumber
\\
			& = \frac{\lambda^{k-kd} e^{-\lambda \theta}}{k!}
			\int_{(\R^d)^k} h_1((o,\lambda^{-1}\bz)) \exp(-\lambda V'(\lambda^{-1} \bz)) d\bz.
			\label{0827a}
		\end{align}
If $\lvvvert \bz \rvvvert \leq \la/2$ then $h_1((o,\la^{-1}\bz)) =
h_\la((o,\bz))=1$, and if $\lvvvert \bz \rvvvert > k \la $ then 
$h_\la((o,\bz))=0$.  Thus provided $\la \geq 4$ so that $\la^{1/2} \leq \la/2$,
	recalling the definition $g_r(\bz):= r^{-1}V'(r\bz)$,
		by (\ref{0827a}) and (\ref{e:alpha}) we have
\begin{align}
	k!| \lambda^{k(d-1)} e^{\lambda \theta} p_{k+1}(\la) - \alpha_{k+1}|
	  \leq  \int_{(B_{\sqrt{\la}})^k} |e^{- g_{1/\la}(\bz)} - e^{- g (\bz)}| d\bz ~~~~~~~~~~~~~~~
			 \nonumber \\
	  + \int_{(B_{\la k})^k \setminus (B_{\sqrt{\la}})^k}
	\exp(- g_{1/\la}(\bz) ) d\bz
			 + \int_{(\R^d)^k \setminus (B_{\sqrt{\la}})^k}
			 e^{-g(\bz)}d\bz. 
			 \label{0915a}
		\end{align}

There is a constant $c = c(d)>0$ such that
$g_r(\bz) \geq c \lvvvert \bz \rvvvert$, 
for all $r>0$ and $\bz \in (\R^d)^k$ with $\lvvvert r \bz\rvvvert \leq k$.
Hence $g_{1/\lambda}(\bz) \geq c \lvvvert \bz \rvvvert$ 
whenever $\lvvvert \bz\rvvvert \leq k \la$, and
		when $\lvvvert \bz \rvvvert > k\la$
		we have $h_1((o,\la^{-1} \bz))=0$.
		Moreover by (\ref{e:gdef}) we also have 
		$g(\bz) = \lim_{r \downarrow 0} g_r(\bz)
		\geq c \lvvvert \bz \rvvvert$ for all $\bz \in
		(\R^d)^k$.
Therefore the second integral in (\ref{0915a}) is bounded by
$
\int_{(B_{\la k})^k \setminus (B_{\sqrt{\la}})^k}
\exp(-c \lvvvert \bz \rvvvert)d\bz,
$
		and hence by $\theta^k (\lambda k)^{dk} \exp(-c \la^{1/2})$.

The third  integral in (\ref{0915a}) is bounded
by $ \int_{(\R^d)^k \setminus (B_{\la^{1/2}}(o))^k}
e^{-c\lvvvert \bz \rvvvert}$, which is $O(\exp (-c' \la^{1/2}))$ as
$\la \to \infty$, for some $c' >0$..

The integrand in the first integral in (\ref{0915a}) can be written as
$ e^{-g(\bz)} |e^{g(\bz)- g_{1/\la}(\bz)} -1|$.  By Lemma \ref{l:QGE}, there
are constants $c'',c'''$ such that for all large $\la$ and  all
$\bz \in (B_{\sqrt{\la}})^k$ we have $|g(\bz) - g_{1/\la}(\bz)| \leq c''
\la^{-1}\lvvvert \bz\rvvvert^2$  
		and thus $|e^{g(\bz)- g_{1/\la}(\bz)} -1 | \leq
		c''' \la^{-1} \lvvvert \bz\rvvvert^2$, so that the first
		integral in (\ref{0915a}) is bounded by
		$$
		c''' \la^{-1} \int_{(\R^d)^k} e^{-g(\bz)} \lvvvert \bz\rvvvert^2
		d\bz,
		$$
	and since $g(\bz)\geq c\lvvvert \bz \rvvvert$ for all $\bz$,
		the last integral is finite and the result is proved.
	\end{proof}

\begin{proof}[Proof of Theorem \ref{t:ptprlim}]
	Let $f:(\R^d)^k \to \R$ be bounded and continuous.
	Write $\bY_{\la}:= (Y_{1,\la},\ldots,Y_{k,\la})$.
	We need to show that 
	$\EE[f(\lambda \bY_{\la})
	|\{ | \cC(\la) | = k+1\}] \to \EE [f(V_1,\ldots,V_k)]$.
	For $(x_1,\ldots,x_k) \in (\R^d)^k$ let
	$f^*((x_1,\ldots,x_k)):= f((x_{(1)},\ldots,x_{(k)}))$,
	where
	$x_{(1)},\ldots,x_{(k)}$ are the
	elements of $\{x_1,\ldots,x_k\}$ taken
	in increasing lexicographic order.
	 Then $f^*$ is a symmetric
	function of $(x_1,\ldots,x_k)$.
	By the multivariate Mecke formula
\begin{align*}
	\E [ f(\lambda \bY_{\la}) \1\{|\cC(\la)|=k+1\} ] 
	& = \frac{\lambda^k}{k!} 
	\int_{(\R^d)^k} f^*(\lambda \bx) h_1((o,\bx)) 
	\exp(-\lambda V((o,\bx))) d\bx
\\
	& = 
	\frac{e^{-\lambda \theta} \lambda^{k(1-d)} }{k!}
	\int_{(\R^d)^k} f^*(\bz) h_1((o,\lambda^{-1} \bz))
	\exp(-g_{1/\lambda}
	(\bz)) d\bz.
\end{align*}
We divide by $\Pr[|\cC(\la)|=k+1]$ and take the $\la \to \infty$ limit.
	By Lemma \ref{l:QGE} the integrand tends to $f^*(\bz) e^{- g(\bz)}$,
	and it is dominated by an integrable function of $\bz$, as in the
	proof of Theorem \ref{t:alpha}. Hence, 
 using (\ref{e:Alim}) and the dominated convergence theorem we obtain that
	\begin{align*}
		\lim_{\la \to \infty} \E[f(\la \bY_\la))|\{|\cC(\la)|=k+1 \}]
		& = (k!\alpha_{k+1})^{-1} \int_{(\R^d)^k} f^*(\bz) \exp(-g(\bz)) d\bz
\\
		& = \EE[f^*((U_1,\ldots,U_k))],
	\end{align*}
	the last line coming from (\ref{e:limPP}) and
	the law of the unconscious statistician.
	Since $f((V_1,\ldots,V_k))=f^*((U_1,\ldots,U_k))$,
	the result is proved.
\end{proof}


\section{First order asymptotics}

 Now we return to the model of finite random geometric graphs, described
in Section \ref{secintro}.
To recall, we let $A \subset \R^d$ be compact with
 $ \vol(A) >0$. If $d \geq 2$ then  assume, for the
 rest of this paper, that $A$ {\em has
a $C^{1,1}$ boundary}
in the sense that for each $x \in \partial A$ there exists
a neighbourhood $U$ of $x$ and a real-valued function $\phi$ that is
defined on an open set
in $\R^{d-1}$ and  differentiable with Lipschitz continuous
partial derivatives,
such that $\partial A \cap U$, such that $\partial A \cap U$, after a rotation,
is the graph of the function $\phi$.
If $d=1$ we assume $A$ is an interval.
This is our smoothness assumption.

Recall that $(X_1,X_2,\ldots)$ is a sequence
of independent random $d$-vectors, each of which has
distribution, denoted $\nu$, with density
function $f$ with support $A$ and with $0< f_0 \leq \fmax <\infty$.
Recall the point processes $\X_n:= \{X_1,\ldots,X_n\}$
and $\Po_{n}:= \{X_1,\ldots,\X_{Z_n}\}$, where
$Z_n$ is Poisson with parameter $n$ (in the second case
$n$ need not be an integer).
Then $\X_n$ is a binomial point process, and $\Po_n$ is
a Poisson point process in $A$ with $\E[\Po_n(dx)]= n f(x)dx$.
In the {\em uniform case} we take $f = f_0 \1_A$,
with $f_0= 1/\vol(A)$.
For $n>0$ let $r=r_n= r(n)>0$ be given, satisfying
(\ref{e:supcri}) and (\ref{e:sublog2}), which we re-state here:
\begin{align}
	\lim_{n \to \infty} nr_n^d = \infty; ~~~~~~~~~~~~
	\limsup_{n \to \infty} n \theta r_n^d/(\log n) <
	\frac{1}{\max(f_0, d (f_0 - f_1/2) ) }.
	\label{e:supcri3}
\end{align}
{\em We shall continue to assume (\ref{e:supcri3}) throughout the
rest of this paper, except for Section \ref{s:sparse}}.

Let $k \in \N$, and recall that $S'_{n,k}$ denotes the 
number of components of $G(\Po_n,r)$ of order $k$
(again, $n$ need not be an integer here).

We give a formula for $I_{n,k} := \E[S'_{n,k}]$.
With $h_r(\cdot)$ and $B_r(\cdot)$ defined at (\ref{e:Brhrdef}),
\begin{align}\label{e:def_Snk}
	S'_{n,k} = \sum_{\ph\subset \Po_n, |\ph|=k} h_r(\ph) 
	\1\{ (\Po_n\setminus \ph) \cap B_r(\ph) 
	= \emptyset \}. 
\end{align}
Hence by the multivariate Mecke formula, 
\begin{align}
	\label{e:exp_Snk}
	I_{n,k}= 
	\EE[S'_{n,k}]= \frac{n ^k}{k!} \int_{A^k} h_r(\bx)
	\exp(- n \nu( B_r(\bx )))
	\nu^k(d\bx). 
\end{align}

It will be useful to have a lower bound on this expectation.

\begin{lemma}
        \label{l:Ilower}
        Let $f_0^+$ be any constant with $f_0^+ >f_0$. Then
        as $n \to \infty$,
        \bea
	n \exp(- \theta f_0^+ nr^d) = o( I_{n,k}).
        \label{e:Iklower}
        \eea
	In particular, $I_{n,k} \to \infty$ as $n \to \infty$.
\end{lemma}
\begin{proof}
Let $f_0 < \alpha < \alpha' < f_0^+$.  By the definition of
$f_0$ and the assumption that $f$ is continuous on $A$, we can (and do)
choose $x_0 \in A^o$ with $f(x_0) < \alpha$ and $f$ continuous at $x_0$. 
        Choose $r_0 >0$ such that $B(x_0,2r_0) \subset A$ and
        $f(y)  \leq \alpha$ for all $y \in B(x_0,2r_0)$.
	Let $\eps \in (0,1/2)$ with $\alpha(1+\eps)^d < \alpha'$.
	If $2r \leq r_0$, and $\bx =(x_1,\ldots,x_k) \in A^k$ with
	$x_1 \in B(x_0,r_0)$, and  $\{x_1,\ldots,x_k\} \subset B(x_1,\eps r)$
	then $B_r(\bx) \subset B_{r(1+\eps)}(x_1)$ and
	$\nu(B_r (\bx)) \leq \alpha \theta (1+ \eps)^d r^d
	< \alpha' \theta r^d$; also $h_r(\bx)=1$.
	Hence by (\ref{e:exp_Snk}), for all large enough $n$,
	\begin{align*}
		I_{n,k} & \geq \frac{n^k}{k!} \int_{B(x_0,r_0)}
		\int_{B(x_1,\eps r)^{k-1}}
		\exp(-n \alpha' \theta r^d)  \nu^{k-1}
	(d(x_2,\ldots,x_k)) \nu(dx_1)
		\\ &
		\geq n  (f_0^k \theta^k r_0^d/k!) (n (\eps r)^d)^{k-1}
		 \exp(- n \alpha' \theta r^d),
	\end{align*}
	so that $(ne^{-n\theta f_0^+r^d})/I_{n,k} $ tends to zero by
	the first part of \eqref{e:supcri3}, and hence (\ref{e:Iklower}).

	Since the assumption (\ref{e:supcri3}) implies $\limsup_{n \to \infty}
	((\theta f_0 n r_n^d)/\log n) <1$, (\ref{e:Iklower})
	implies $I_{n,k} \to \infty$ as $n \to \infty$.
\end{proof}

We next determine the asymptotics for $\E[S_{n,k}]$.
\begin{proposition}[Mean $k$-cluster count in binomial process]
	\label{p:Snmeanbin}
There exists $c>0$ such that as $n \to \infty$,
	 with $I_{n,k}$ defined at (\ref{e:exp_Snk}), 
	 we have
	\bean
	\EE[S_{n,k}]= 
	I_{n,k}( 1+ O(e^{-c nr^d}) ) .
	 \eean
\end{proposition}
\begin{proof}
	 Observe that
	 \bea
	 \EE[S_{n,k}] = \binom{n}{k} \int_{A^k}
		 h_r(\bx) (1- \nu(B_r(\bx)))^{n-k} \nu^k(d \bx).
		 \label{e:ESn}
	 \eea
 We compare this with (\ref{e:exp_Snk}).
		 Let $\bx \in A^k$ and set $p= \nu(B_r(\bx))$. Then
 \begin{align*}
		 |e^{-np} - (1-p)^{n-k}| \leq e^{-np}|1- (e^p (1-p))^n|
			 + (1-p)^{n-k}|(1-p)^k -1|.
	 \end{align*}
	 Now
	$e^p(1-p) \leq e^p e^{-p} = 1$, and
 moreover $e^p(1-p) \geq (1+p)(1-p) = 1-p^2$.  Therefore
$$
		 0 \leq e^{-np}(1- (e^p (1-p))^n) \leq e^{-np}(1-
		 (1- p^2)^n) \leq np^2 e^{-np}.
$$
	Moreover $1-(1-p)^k = O(r^d)$.
Thus $|e^{-np}-(1-p)^{n-k}|=O(nr^{2d})$,
 uniformly over $\bx \in A^k$.
	Also $\int_{A^k} h_r(\bx) \nu^k(d\bx) = O(r^{d(k-1)})$.
 Therefore 
	 \begin{align*}
		 \binom{n}{k}	 \int_{A^k}h_r(\bx)
		 |e^{-n \nu(B_r(\bx))}- (1- \nu(B_r(\bx)))^{n-k}|
		 \nu^k(d\bx)
		 = O(n^k r^{d(k-1)}(nr^{2d})),
	 \end{align*}
	 which is $O((\log n)^{k+1} )$
	 by (\ref{e:supcri3}).
	 By Lemma \ref{l:Ilower} and (\ref{e:supcri3}) there exists
	 $\delta >0$ such that $I_{n,k} = \Omega(n^{2 \delta})$. Hence
	 by (\ref{e:supcri3}) again there exists $c>0$ such
	 that $(\log n)^{k+1}/I_{n,k} = O(n^{-\delta}) = O(e^{-c nr^d})$.
	 Thus
	 the last display is
	 $O(e^{-c n r^d}I_n)$ for some $c>0$. Moreover
	 \begin{align*}
		 \left| \frac{n^k}{k!} - \binom{n}{k} \right|
		 \int_{A^k}
		 h_r(\bx) e^{-n \nu(B_r(\bx))} \nu^k(d \bx)
		 = O(n^{-1} I_{n,k}),
	 \end{align*}
	 which is $O(e^{-c nr^d} I_{n,k})$ for some other $c>0$. Combining
	 these estimates, and using (\ref{e:ESn}) and (\ref{e:exp_Snk}),
	 we obtain that $|\E[S_{n,k}] - I_{n,k}| = O( e^{-cnr^d}I_{n,k})$ 
	 for some $c>0$, as required.
	 \end{proof}

We now verify the asymptotic expression (\ref{e:Inasymp})  for $I_{n,k}$
in the uniform case,
with a bound on the 
rate of convergence.
 \begin{theorem} [Asymptotic for $I_{n,k}$ in the uniform case]
	 \label{t:Ilim}
	 Let $k \in \N$. In the uniform case,
\bea
	 n^{-1} (f_0 n r^d)^{(k-1)(d-1)} e^{f_0 \theta n r^d}
I_{n,k} = k^{-1} \alpha_k
	 + O( (nr^d)^{-1})~~~~ {\rm as} ~~ n \to \infty.
\label{e:Ilim}
\eea
 \end{theorem}
 Before proving this, we give  two geometrical lemmas that we shall
 use repeatedly later on to deal with boundary effects.

 \begin{definition}[Sphere condition]
        For $z \in \partial A$ let $\hat n_z$ be the unit normal to $\partial A$ at $z$ pointing inside $A$.

        Given $\tau \geq 0$,
        let us say $\tau $ satisfies the {\em sphere condition}
         for $A$ if, for all $x \in \partial A$,
        we have $B(x+ \tau \hat n_x , \tau) \subset A$
        and  $B(x -\tau \hat n_x, \tau) \cap A = \{x\}$.

Let $\tau(A)$ denote the supremum of the set of all $\tau$ satisfying
the sphere condition for $A$.
\end{definition}
\begin{lemma}[Sphere condition lemma]
        $\tau(A) >0$; that is,
        there exists a constant $\tau >0$ such that
        $\tau$ satisfies the sphere condition
         for $A$.
\end{lemma}
\begin{proof}
        See     \cite[Lemma 7]{sphere-condition-lp}.
\end{proof}
\begin{remark}
        {\rm
        (i) If $0 < \tau < \tau'$ and $\tau'$ satisfies the sphere condition
 for $A$, then so does  $\tau$.

        (ii) If $x \in \R^d$ with
        $\dist(x ,\partial A) < \tau(A)$,
        then $x$ has a unique closest point in $\partial A$.
        }
\end{remark}

 \begin{lemma}
	 \label{l:bdy_est}
Writing $\kappa(\partial A,r)$ for the number
 of balls of radius $r$ required to cover $\partial A$, we have
\bea
\limsup_{s \downarrow 0} s^{d-1}\kappa(\partial A,s) < \infty.
\label{e:covering}
\eea
	 Moreover
 \bea
 \label{e:volLB}
\liminf_{s \downarrow 0} ( s^{-d} \inf_{x \in A} \vol(A \cap B_s(x)) )
	 \geq \theta /2.
\eea
	 Finally, as $s \downarrow 0$ we have
	 $\limsup(s^{-1}\vol(A \setminus A^{(s)})) < \infty$. 
 \end{lemma}
 \begin{proof} If $d=1$ we are assuming $A$ is a compact interval, and
	 all the assertions of the lemma are clear. 

Suppose $d \geq 2$, and let $x \in \partial A$.
	  By our smoothness assumption,
	 there is a a rotation about $x$ 
	 and a neighbourhood $V$ of $x$
	 such that after rotation
	 the set $(\partial A) \cap V $ equals $ \{(u,\phi(u)): u \in U\}$
	 where $U$ is an open set in $\R^{d-1}$
	 and $\phi: U\to \R$ is continuously differentiable with 
	 Lipschitz derivatives. By a compactness argument,
	 for \eqref{e:covering} it
	 suffices to prove that $\limsup_{s \downarrow 0} s^{d-1} 
	 \kappa((\partial A) \cap V,s) < \infty$.

	 Without loss of generality we can assume that $x=o$ and
	 our roatition is the identity map, and that $\nabla \phi(x)
	 =0$. Also  since $\phi$ is continuously differentiable
	 we can also assume
	 (by taking a smaller neighourhood $V$ if needed)
	 that $U = B(o,\delta)$ for some $\delta >0$ and that
	 $|\langle \nabla \phi(u) ,e\rangle| \leq 1$ for all
	 $u \in U$ and all
	 unit vectors $e \in \R^{d-1}$.
	 Then by the Intermediate Value theorem, for
	 all $u,v \in U$ we have for some $w \in [u,v]$
	 that
	 $$|\phi(v) -\phi(u)| = | \langle v-u, \nabla \phi(w) \rangle | \leq
	 \|v-u\|.
	 $$
	 Given $s >0$, we can and do cover $U$ by $(d-1)$-dimensional
	 balls $B_1,\ldots,B_{m(s)}$  of radius $s/2$ centred in $U$, with 
	 the centre of $B_i$ denoted $u_i$ for each $i$, and
	 and with $m(s) = O(s^{1-d})$ as $s \downarrow 0$. Then for
	 $i \leq m(s)$ and $v \in B_i$,
	 $$
	 \|(v,\phi(v) ) - (u_i,\phi(u_i)) \| \leq \|v-u_i\| + \|\phi(v)
	 -\phi(u_i)\| \leq (s/2) + (s/2) =s.
	 $$
	 Therefore the balls $B((u_i,\phi(u_i)),s)$, $1 \leq i \leq m(i)$,
	 cover $\partial A \cap V$, and this gives us \eqref{e:covering}. 

	 For \eqref{e:volLB}, we use
	 \cite[Lemma 3.4]{HPY25}.
	 By that result,
	 for $0 < s < \tau(A)$ and $x \in A \setminus A^{(s)}$,  
	 $$
                \left|
		\vol( A \cap B_s(x) ) - ((\theta_d/2) + h(\dist(x,\partial A)/s))s^d
                \right|
                \leq
		\frac{2 \theta_{d-1} s^{d+1}}{\tau(A)},
		$$
		where $h(a):= \vol(B_1(o) \cap ([0,a] \times \R^{d-1}).$
		In particular $h(a) \geq 0$ so that 
		$$
		s^{-d}\vol(A \cap B_s(x)) \geq
		(\theta_d/2) - 2 \theta_{d-1} s/\tau(A),
		$$
		and \eqref{e:volLB} follows.



	 For the last assertion, given $s>0$, set $\kappa_s:=
	 \kappa(\partial A,s)$. By definition of $\kappa_s$,
	 we can find $x_1,\ldots,x_{\kappa_s} \in \R^d$ such that
	 $\partial A \subset \cup_{i=1}^{\kappa_s} B(x_i,s)$.
	 Then $A \setminus A^{(s)} \subset \cup_{i=1}^{\kappa_s} B(x_i,2s)$.
	 Also $\kappa_s = O(s^{1-d})$ as $s \downarrow 0$,
	 by (\ref{e:covering}). 
	 Hence by the union bound,  
	 $$
	 \vol(A \setminus A^{(s)} ) \leq \sum_{i=1}^{\kappa_s}
	 \vol(B(x_i,2s)) 
	 = O( s),
	 $$
	 which is the last assertion.
 \end{proof}

\begin{proof}[Proof of  Theorem \ref{t:Ilim}]
	Assume $\nu$ is uniform on $A$.
	By definition $I_{n,k} = \E[S'_{n,k}]$. Hence,
	by the (univariate) Mecke formula,
	writing
	$\Po_n^x$ for $\Po_n \cup \{x\}$,
	we have
\begin{align}
	I_{n,k} = \frac{n}{k} \int_A \Pr[|\cC_{r}(x,\Po_n^x)| =k] f_0 dx,
	\label{e:Ink}
\end{align}
	By the multivariate Mecke formula, setting  
	 $V_r(\bx):= \vol( B_r(\bx))$ for $\bx \in (\R^d)^k$,
	for $x \in A^{(kr)}$ we have 
$$
\Pr[|\cC_{r}(x,\Po_n^x) |=k ] = \frac{(nf_0)^{k-1}}{(k-1)!} 
\int_{(\R^d)^{k-1}}
	h_r((x,\bx)) \exp(-n f_0 V_r((x,\bx))) d \bx,
$$
	where we used the fact that if $x \in A^{(kr)}$
	and $h_r((x,\bx)) =1$ then $B_r((x,\bx)) \subset A$.
Hence by translation invariance, for  $x \in A^{(kr)}$, writing 
	$h_r(o,\bx)$ for $h_r((o,\bx))$ and
	$V_r(o,\bx)$ for $V_r((o,\bx))$ we have
\begin{align*}
	\Pr[|\cC_{r}(x,\Po_n^x)| =k ] & = 
	\frac{(nf_0)^{k-1}}{(k-1)!} 
	\int_{(\R^d)^{k-1}}
h_r(o,\bx) \exp(-n f_0 V_r(o,\bx))
d \bx
\\
	& = \frac{(nf_0 r^d)^{k-1}}{(k-1)!} 
	\int_{(\R^d)^{k-1}}
h_r(o,r\by) \exp(-n f_0 V_r(o,r \by))
d \by
\\
	& = \frac{(nf_0 r^d)^{k-1}}{(k-1)!} 
	\int_{(\R^d)^{k-1}}
h_1(o,\by) \exp(-n f_0 r^d V_1(o, \by))
	d \by
\end{align*}
and by (\ref{e:pk}), this equals $ p_{k}(f_0n r^d)$.
Therefore using (\ref{e:Alim}) and the last part of Lemma
	\ref{l:bdy_est} we obtain that
as $n \to \infty$,
\begin{align}
\int_{A^{(kr)}} \Pr[|\cC_{r}(x,\Po_n^x)|=k] f_0 dx
	= (1+O(r))
  \alpha_k (f_0nr^d)^{(1-k)(d-1)} e^{-\theta f_0 n r^d }
	(1+ O((nr^d)^{-1})).
	\label{0802b}
\end{align}

	To deal with $x \in A \setminus A^{(kr)}$,
	let $f_0^- \in (0,f_0)$. 
	 	By Lemma \ref{l:bdy_est}
	we have for all large enough $n$ and all $x \in A$
	that $\nu(B_r(x)) \geq  \theta f_0^- r^d/2$. 
	Therefore using the multivariate Mecke formula
	we have for all large enough $n$ and all $x \in A$ that
	$$
	\Pr[|\cC_r(x,\Po_n^x)| =k] 
	\leq
	\frac{(n f_0)^{k-1}}{(k-1)!}
	\int_{A^{k-1}}
	h_r(x,\bx) \exp(-n \theta f_0^- r^d/2)   
	d\bx,
	$$
	which is $O((nr^d)^{k-1} \exp(-n \theta f_0^- r^d/2))$,
	uniformly over $x \in A$.
	Hence by the last part of Lemma
	\ref{l:bdy_est} we obtain that
	$$
\int_{A \setminus A^{(kr)}} \Pr[|\cC_{r}(x,\Po_n^x)|=k] f_0 dx
= O(r (nr^d)^{k-1} \exp(-n \theta f_0^- r^d/2) ),
$$
which is negligible compared to the 
right hand side of (\ref{0802b}) by (\ref{e:supcri3}), provided
we take $f_0^-$ close enough to $f_0$ (here $f_0=f_1$ so
(\ref{e:supcri3}) implies $\limsup_{n \to \infty} (n \theta f_0 r_n^d/\log n)
< 2/d$). Thus
using (\ref{e:Ink}) we obtain (\ref{e:Ilim}). 
\end{proof}

In the non-uniform case, we do not have such precise
asymptotics for $I_{n,k}$. However we do have the following,
which formalises back-of-the-envelope calculations
suggesting $I_{n,k} \approx n \exp(-n \theta f_0 r_n^d)$.
\begin{theorem}[{\rm Limiting behaviour of $I_{n,k}$ in the non-uniform
	case}]
	\label{t:Inonunif}
	Let $k \in \N$. Then
	\bea
	\lim_{n \to \infty} \left( (n r^d)^{-1} \log (I_{n,k}/n) \right)
	= - \theta f_0.
	\label{e:Ilimlog}
	\eea
\end{theorem}
\begin{proof}
	By Lemma \ref{l:Ilower}, given $f_0^+ > f_0$ we have
	for large $n$ that $I_{n,k} \geq n \exp(-n \theta f_0^+ r^d)$,
	so that $\log(I_{n,k}/n) \geq - n \theta f_0^+ r^d$, and hence
	$$
	\liminf_{n \to \infty} \left( (n  r^d)^{-1} \log (I_{n,k}/n) \right)
	\geq - \theta f_0.
	$$
	For a bound the other way, note that $\nu(B_r(x) ) \geq \theta f_0 r^d$
	for $x \in A^{(r)}$. Also using Lemma \ref{l:bdy_est},
	given $f_1^- < f_1$ we have for $n $ large enough that
	$\nu(B_r(x)) \geq (\theta/2) f_1^- r^d$ for all $x\in A$. Therefore
	using the last part of Lemma \ref{l:bdy_est}, we have
	that
	\begin{align}
		I_{n,k} \leq & n \nu(A^{(r)}) (n \fmax \theta ((k-1)r)^d)^{k-1}
	\exp(-n \theta f_0 r^d)
	\nonumber	\\
		& + O(nr (nr^d)^{k-1} \exp(-n (\theta/2)
	f_1^- r^d) ).
	\label{e:Ilb}
		\end{align}
	The second term in the right-hand side of (\ref{e:Ilb}),
	divided by the first term,
	is bounded by a constant times 
	$r \exp(n \theta r^d(f_0 - f^-_1/2))$, and using
	(\ref{e:supcri3}) we can show this tends to zero, provided
	$f_1^-$ is chosen close enough to $f_1$.
	Hence by (\ref{e:Ilb})
	we can deduce that
	$$
	\limsup_{n \to \infty} \left( (n  r^d)^{-1} \log (I_{n,k}/n) \right)
	\leq - \theta f_0,
	$$
	and hence (\ref{e:Ilimlog}).
\end{proof}
\section{Variance asymptotics and laws of large numbers}
\label{s:t3}
In this section,
we  show that the variance of the $k$-cluster count
$S_{n,k}$ and its expectation are asymptotically equivalent, with explicit rates, and likewise for $S'_{n,k}$. 
As mentioned in Section \ref{secintro}, this is enough to yield
 weak laws of large numbers for $S_{n,k}$ and $S'_{n,k}$ but we shall
 also derive 
 concentration results which yield strong laws of large numbers.
\subsection{Variance asymptotics in the Poisson setting}
\begin{proposition}
\label{p:var_k}
	Let $k\in\NN$. Then there exists $c >0$ such that as $n \to \infty$,
 \begin{align*}
	 \Var[S'_{n,k}]= \EE[S'_{n,k}](1+O(e^{-cnr^d}))
	 = I_{n,k} (1+O(e^{-cnr^d})).
 \end{align*}
\end{proposition}
\begin{proof}
	Since $S'_{n,k}(S'_{n,k}-1)$ is the number of ordered
	pairs of distinct $k$-clusters in $G(\Po_n,r)$,
	with $h(\cdot):= h_r(\cdot)$ as given at (\ref{e:Brhrdef}),
	we can write $S'_{n,k}(S'_{n,k}-1)$ as follows
\begin{align*}
	\sum_{\ph, \psi\subset \Po_n, |\ph|=|\psi|=k} 
	h(\ph)h(\psi)\1\{ (\Po_n\setminus \ph\setminus \psi) 
	\cap B_r (\ph\cup\psi)=\emptyset,  \dist(\ph,\psi) > r \}.
\end{align*}
Indeed, 
	for distinct $k$-subsets $\psi, \ph$ of $\Po_n$, their distance must be larger than $r$ in order that they are both connected components of order $k$, for otherwise $\psi$ and $\ph$ are connected with $|\psi\cup\ph|>k$. 
	Thus by the multivariate Mecke equation
	$\EE[(S'_{n,k})^2]-\EE[S'_{n,k}]$ equals
\begin{align*}
	\frac{n^{2k}}{k!k!} 
	\int_{A^{k}} \int_{A^k} 
	h(\bx)h(\by) \exp(-n \nu(B_r(\bx,\by)))
	\1\{\dist(\bx,\by) > r\} \nu^k(d\by) \nu^k( d\bx),
\end{align*}
	where $B_r(\bx,\by):= B_r(\bx) \cup B_r(\by)$.
	We compare this integral with 
\begin{align*}
	\EE[S'_{n,k}]^2 = \frac{n^{2k}}{k!k!} 
	\int_{A^{k}} \int_{A^k} 
	h(\bx)h(\by) \exp(-n[ \nu(B_r(\bx)) +  \nu (B_r(\by))])
	\nu^k(d\by) \nu^k(d\bx),
\end{align*}
	which comes from (\ref{e:exp_Snk}).
	Observe that  $\nu(B_r(\bx,\by)) = \nu(B_r(\bx)) + \nu(B_r(\by))$
	whenever $\dist(\bx,\by) > 2r$.
	Therefore $|\Var[S'_{n,k}]-\EE[S'_{n,k}]|\le J_{1,n}+J_{2,n}$, 
	where
\begin{align}
	J_{1,n}&:=\frac{n^{2k}}{k!k!} \int_{A^{k}} \int_{A^k}
	h(\bx)h(\by) \exp(-n \nu(B_r(\bx,\by)))
	\1\{\dist(\bx,\by)\in (r,2r] \} 
	\nu^{k}(d\by) \nu^{k}(d\bx);
	\label{e:J1def}
	\\
	J_{2,n}&:= \frac{n^{2k}}{k!k!} \int_{A^{2k}}
	h(\bx)h(\by) \exp(-n[\nu(B_r(\bx))+\nu(B_r(\by))])
	\1\{\dist(\bx,\by) \leq 2r \} \nu^{2k}(d(\bx, \by)).
	\label{e:J2def}
\end{align}
We estimate $J_{1,n}$ in Lemma \ref{l:J1} below.
For $J_{2,n}$, note that by Lemma \ref{l:bdy_est}  there exists $n_0$ 
such that for all $n \geq n_0$ and all
$\by \in A^k$
we have $\nu(B_r(\by)) \geq (\theta/3) f_0 r^d$.
Hence for $\bx \in A^k$ we have
$$
\frac{n^k}{k!} \int_{A^k} h(\by)
\exp(-n \nu(B_r(\by)))
\1\{\dist(\bx,\by) \leq 2r\} \nu^k(d \by)
=O( (nr^d)^k \exp(-n (\theta/3) f_0  r^d)),
$$
and hence 
$J_{2,n}
= O( I_{n,k} (nr^d)^k\exp(-n f_0 (\theta/3) r^d))$, which
is $O(I_{n,k} \exp(-n f_0 (\theta/4)r^d))$.
Combined with Lemma \ref{l:J1}, this completes the proof.
\end{proof}
\begin{lemma}
	\label{l:J1}
	Let $J_{1,n}$ be given by (\ref{e:J1def}).
	There exists $c>0$ such that 
	$J_{1,n} = O(e^{-cnr^d}I_{n,k})$ as $n \to \infty$.
\end{lemma}
\begin{proof}
Let us write $\bx \prec \by$ 
	if there exists
	$y \in \{y_1,\ldots,y_k\}$ such that  all points of
	$\{x_1,\ldots,x_k\}$ precede $y$ in the lexicographic ordering
	(or in other words,
	 the rightmost entry of $\bx$ lies to the
	left of the rightmost entry of $\by$).
	Since the integrand in $J_1$ 
	is symmetric in $\bx$ and $\by$, $J_1$ is twice the same expression
	with the inner integral restricted to $\by$ satisfying $\bx \prec \by$.

	Suppose $\bx, \by \in (\R^d)^k$ with $h(\bx)=h(\by)=1$
	and $\dist(\by,\bx)\in (r,2r]$. 
	These conditions imply that $\by \in (B_{2kr}(x_1))^k$, where
	we write $\bx = (x_1,\ldots,x_k)$.  
	Thus, if also $\bx \in (A^{(9kr)})^k$ 
	then $\by \in (A^{(7kr)})^k$,
	and if moreover $\bx \prec \by$,
	by considering the right half of the ball of radius $r$ centred
	on the right-most entry of $\by$, we see that
	$\nu( B_r(\by) \setminus  B_r(\bx)) \geq (\theta f_0/2)r^d$.
	Therefore since
	$B_r(\bx,\by) = B_r(\bx) \cup (B_r(\by) \setminus B_r(\bx))$
	and this is a disjoint union, 

	\begin{align*}
		J_{1,n} \le \frac{2n^{k}}{k!} \int_{(A^{(9kr)})^k} h(\bx)
		\exp(-n \nu(B_r(\bx)))
		a(\bx) \nu^k(d\bx) + z_n,
\end{align*}
where
\bea
	a(\bx) :=\frac{n^k}{k!} \int_{A^k} e^{-(\theta f_0/2)nr^d} h(\by) \1\{
	\dist(\bx,\by)\in (r,2r], \bx \prec \by  \} \nu^k(d\by),
\label{0729a}
\eea
 and $z_n$ is a remainder term from $\bx$ near $\partial A$ 
 (to be dealt with later).

	As already mentioned, if $\bx, \by \in A^k$ with $h(\bx)=h(\by)=1$ and
	$r < \dist(\bx,\by) \leq 2r$, then $\by \in ( B_{2kr}(x_1))^k$, and 
	consequently
\begin{align*}
	a(\bx) \le \frac{(n\fmax)^k}{k!} e^{-(\theta f_0/2) nr^d}
	\theta^{k}  (2kr)^{kd} 
	= O( (n r^d)^k e^{- (\theta f_0/2)nr^d})
	= O( e^{- (\theta f_0/3)nr^d}),
\end{align*}
uniformly over $\bx \in A^k$. Therefore by
 virtue of \eqref{e:exp_Snk}, we have
$J_{1,n} - z_n =O( e^{-(\theta f_0/3)n r^d} I_{n,k}) $.

For the remainder term, note that given $f_1^-< f_1$,
by Lemma \ref{l:bdy_est} and the assumed continuity of $f|_A$,
there exists $n_0$ such
that for all $n \geq n_0$ and
 all $\bx \in A^k \setminus (A^{(9r)})^k$, we
 have $\nu(B_r(\bx)) \geq n (\theta f_1^-/2) r^d$.
 Also if $\bx = (x_1,\ldots,x_k) \in
A^k \setminus (A^{(9kr)})^k$ then at least one of $x_1,\ldots,x_k $
lies in $A \setminus A^{(9kr)}$. Introducing a factor of $k$ we
can assume this is $x_1$, and hence
deduce that
\bean
|z_n| \leq \frac{2 (n\fmax)^{2k}}{(k-1)!}
\int_{A \setminus A^{(9kr)}} 
(\theta (2 k r)^d)^{2k-1}   
d x_1 e^{-n\theta f_1^- r^d/2}= O(n r (nr^d)^{2k-1}
e^{-n \theta f_1^- r^d/2 }). 
\eean
Combined with (\ref{e:Iklower}) this shows that
$z_n/I_{n,k} = O( r \exp(n \theta r^d(f_0^+ -f_1^-/2)))$,
for any $f_0^+ > f_0$.
Using (\ref{e:supcri3}), we can choose $\eps >0$, $f_0^+ >f_0$ and $f_1^- < f_1$
in such a way that  this upper bound implies  $z_n/I_{n,k} =O(rn^{(1/d)-\eps})$,
and using (\ref{e:supcri3}) again we see that this is $O(e^{-c n r^d})$ for
some $c>0$. Thus $J_{1,n} = O(e^{-cnr^d}I_{n,k})$ for some $c >0$.
\end{proof}


\subsection{Variance asymptotics in the binomial case}

\begin{proposition}
\label{p:var_iso_bin}
There exists $c>0$ such that as $n \to \infty$,
	\bea
	\Var[S_{n,k}]=\EE[S_{n,k}](1+O(e^{-cnr^d})).
	\label{e:varSn}
	\eea
\end{proposition}
\begin{proof}
	Let $\prec$ denote the lexicographic ordering on $\R^d$.  Write 
	$S_{n,k}= \sum_{i=1}^n \xi_i$,  here  taking
	\begin{align}
	\xi_i:= \1\{|\cC_r(X_i,\cX_n)|=k, X_i \prec y ~ \forall
	y \in \cC_r(X_i,\cX_n) \setminus \{X_i\}\}.
	\label{e:xidef}
		\end{align}
	Then
	 $\EE[S_{n,k}] = n \EE[\xi_1]$, and
	 $
	 \EE[S_{n,k}^2]= n\EE[ \xi_1^2]+ n(n-1) \EE[\xi_1\xi_2],
	 $
	 so that $\EE[S_{n,k}^2]= \EE[S_{n,k}]+ n(n-1) \EE[\xi_1\xi_2]$.
	 Hence,
	 \bea
	\Var[S_{n,k}]- \EE[S_{n,k}]= - n \EE[\xi_1\xi_2] + n^2
	(\EE[\xi_1 \xi_2] - \EE[\xi_1] \EE[\xi_2]).
	\label{0729b}
	\eea
	 In view of 
	Proposition \ref{p:Snmeanbin}
	 we need to show the right hand
	 side of (\ref{0729b}) is $O(e^{-cnr^d}I_{n,k})$.

	For $\bx = (x_1,\ldots,x_k) \in (\R^d)^k$,
	recall from (\ref{e:Brhrdef})
	that $h^*_r(\bx)$ is the indicator variable
	of the event that $G(\{x_1,\ldots,x_k\},r)$ is connected
	and moreover $x_1 \prec x_i$ for $i=2,\ldots,k$.  Then
	\begin{align}
	\EE[\xi_1]=  \binom{n-1}{k-1}
	\int_{A^{k}} h^*_r(\bx) (1- \nu(B_r(\bx)))^{n-k} \nu^{k}
	(d \bx).
		\label{e:Exi1}
	\end{align}
	 Consider the first term in the right hand side of (\ref{0729b}).
	 Observe that if $\xi_1 = \xi_2 =1$ then
	 $\cC_r(X_2,\X_n) \cap B_r(\cC_r(X_1,\X_n) ) = \emptyset$.
	 Hence, using the bound $1-t \leq e^{-t}$
	 for all $t \geq 0$ and writing $B_r(\bx,\by):= B_r(\bx) \cup 
	 B_r(\by)$, we have that $n \EE[\xi_1 \xi_2]$ equals
	 \begin{align}
		 & n \binom{n-2}{k-1} \binom{n-1-k}{k-1} 
		 \int_{A^k}
		 \int_{(A \setminus B_r(\bx))^k }
		 h_r^*(\bx) h_r^*(\by)
		 (1- \nu(B_r(\bx,\by)))^{n-2k} 
		 \nu^{k}(d \by)
		 \nu^{k}(d \bx)
		 \label{e:Exi1xi2}
		 \\
		 & \leq \frac{2n^{2k-1}}{(k-1)!^2}
		 \int_{A^k}
		 \int_{(A \setminus B_r(\bx))^k }
		 h_r^*(\bx) h_r^*(\by) 
		 \exp(- n \nu(B_r(\bx,\by)) ) 
		 \nu^k(d\by) \nu^k(d\bx),
		 \label{0729d}
	 \end{align}
	 provided $n$ is large enough so that $e^{2 k \nu(B_r(\bx,
	 \by))} \leq 2$ for all $\bx,\by \in A^k$.

Let $\alpha < \min(f_0,f_1/2)$.  Using Lemma \ref{l:bdy_est},
we have for all large enough $n$ and all
$\bx \in A^k$, $\by \in (A \setminus B_{2r}(\bx))^k$  that $\nu( B_r(\bx,\by)) \geq
	 \nu(B_r(\bx)) + \theta \alpha r^d$.
 Hence for such $(\bx,\by)$  the  integrand in (\ref{0729d}) is bounded by
 $h^*_r(\bx)h^*_r(\by) e^{-n \nu(B_r(\bx))} e^{-\alpha \theta n r^d}$. Since
 $\int_{A^k} h^*_r(\by)\nu^k(d\by) = O(r^{d(k-1)})$, and $I_{n,k} =
 \frac{n^k}{(k-1)!} \int_{A^k} h^*_r(\bx) e^{-n \nu(B_r(\bx))} \nu^k(d\bx)$
 by the multivariate Mecke formula,
 \begin{align*}
 n^{2k-1} \int_{A^k} \int_{(A \setminus B_{2r}(\bx))^k }
h_r^*(\bx) h_r^*(\by) e^{- n \nu(B_r(\bx,\by)) } \nu^k(d\by) \nu^k(d\bx)
 = O( (nr^d)^{k-1} e^{-n\theta \alpha r^d} I_{n,k}). 
	 \end{align*}
 On the other hand, if $\by \in A^k \setminus (A \setminus B_{2r}(\bx))^k$,
and $h_r^*(\by) =1$, then $\by \in (B_{2kr}(x_1) )^k$, and hence
$\int_{A^k \setminus (A \setminus B_{2r}(\bx))^k} h_r^*(\by) \nu^k(d\by)
	 = O(r^{dk})$, and $\nu(B_r(\bx,\by)) \geq \nu(B_r(\bx))$, so that
\begin{align*}
   n^{2k-1} \int_{A^k} \int_{A^k \setminus (A \setminus B_{2r}(\bx))^k }
h_r^*(\bx) h_r^*(\by) e^{- n \nu(B_r(\bx,\by)) } \nu^k(d\by) \nu^k(d\bx)
		 = O( r^d (nr^d)^{k-1}  I_{n,k}), 
\end{align*}
 which is $O(e^{-cnr^d}I_{n,k})$ for some $c>0$ by (\ref{e:supcri3}). 
Thus the first term in the right hand side of (\ref{0729b})
	  is $O( e^{- c n r^d}I_{n,k})$ for some $c>0$.

Now consider the second term in the right hand side of (\ref{0729b}).
Given $n$, for $\bx,\by \in (\R^d)^k$
let $p_\bx = \nu(B_r(\bx))$ and $p_{\bx,\by}:= \nu(B_r(\bx,\by))$.
By (\ref{e:Exi1}) and (\ref{e:Exi1xi2}),
\begin{align}
	\EE[\xi_1 \xi_2] - \EE[\xi_1] \EE[\xi_2]
	= \int_{A^k} \int_{(A \setminus B_{2r}(\bx))^k}
	\left[\binom{n-2}{k-1} \binom{n-1-k}{k-1} (1-p_\bx -p_\by)^{n-2k} 
	\right.
	\nonumber \\
	\left.
	- \binom{n-1}{k-1}^2 (1-p_\bx)^{n-k}(1-p_\by)^{n-k} \right]
	h^*_r(\bx) h^*_r(\by) \nu^k(d\by) \nu^k(d\bx)
	\nonumber
	\\
	+ \binom{n-2}{k-1} \binom{n-1-k}{k-1}
	\int_{A^k} \int_{(A \setminus B_r(\bx))^k \setminus 
	(A \setminus B_{2r}(\bx))^k } h^*_r(\bx) h^*_r(\by)  (1-p_{\bx,\by})^{n-2k}
	\nu^k(d\by) \nu^k(d\bx)
	\nonumber \\
	- \binom{n-1}{k-1}^2 \int_{A^k}
	\int_{A^k \setminus (A \setminus B_{2r}(\bx))^k }  
	h^*_r(\bx) h^*_r(\by)
	 (1-p_\bx)^{n-k}(1-p_\by)^{n-k} \nu^k(d\by) \nu^k(d\bx).
	 \label{0729e}
\end{align}
Note that $  p_\bx \leq k \fmax \theta r^d$ for all $\bx \in A^k $.
Writing $p$ for $p_\bx$ and $q$ for $p_\by$, and using the fact that
as $n \to \infty$ with $k $ fixed,
$$
\binom{n-1}{k-1}^2 \div \left[ \binom{n-2}{k-1} \binom{n-1-k}{k-1} \right]
= 1 + O(n^{-1}) ,
$$
we estimate 
the integrand in the first term in the right hand side of (\ref{0729e})
as follows:
\begin{align}
	& \binom{n-2}{k-1} \binom{n-1-k}{k-1}
	(1-p-q)^{n-2k} - \binom{n-1}{k-1}^2 
	(1- p -q+ pq)^{n-k} 
	\nonumber \\
	& =
	 \binom{n-2}{k-1} \binom{n-1-k}{k-1}
	(1-p-q)^{n-k}\left[ (1-p-q)^{-k} - (1 +O(n^{-1}))
	\left( 1 + \frac{pq}{1-p-q} \right)^{n-k}
\right]
\nonumber \\
	&\leq \frac{n^{2k-2}}{(k-1)!^2} \exp \left(  (n-k) \log (1-p-q) \right) 
	\times [ (n-k) pq + O(r^d) + O(n^{-1})]
\nonumber \\
	&\leq \frac{n^{2k-1}}{(k-1)!^2} e^{-np} (k \fmax \theta)^2r^{2d}  e^{-nq}(1+ o(1)).
	\label{0730a}
	\end{align}
By Lemma \ref{l:bdy_est}, for all large enough $n$ and all $\by \in A^k$
we have $q=p_\by \geq  (f_0 \theta/3) r^d$.  Multiplying the 
expression at (\ref{0730a})  by $h^*_r(\bx)h^*_r(\by)$ and integrating over
$\by \in (A
\setminus B_{2r}(\bx))^k$ yields an expression bounded by
\bean
c' n^{k} (nr^d)^{k-1} \exp(- (f_0 \theta/3) nr^d) \exp(-n p_\bx) r^{2d} h^*_r(\bx),
\eean
for some constant $c'$ independent of $n$.
Recalling $I_{n,k} = \frac{n^k}{(k-1)!} \int e^{-n p_\bx} h^*_r(\bx) \nu^k 
(d\bx)$, we find that the absolute value of the
first term in the right hand side of (\ref{0729e}), multiplied
by $n^2$, is bounded above by
\bean
c' (k-1)!
(nr^d)^{k+1} I_{n,k} \exp(- (f_0 \theta/3) n r^d)
= O(I_{n,k} \exp(-(f_0 \theta /4) n r^d)).
\eean

We turn to the second term in the right hand side of (\ref{0729e}).
By the bound $1-t \leq e^{-t}$, for large $n$ this term, multiplied
by $n^2$, is bounded by 
$$
\frac{2 n^{2k}}{(k-1)!^2} \int_{A^k}
\int_{(A \setminus B_r(\bx))^k \setminus (A \setminus B_{2r}(\bx))^k}
h_r^*(\bx) h_r^*(\by) \exp(-n \nu(B_r(\bx,\by))) \nu^k(d\bx) \nu^k(d\by).
$$
By
the multivariate Mecke formula, the last 
displayed expression  equals
 twice the expected number of ordered pairs $(\varphi,\psi)$
of distinct $k$-clusters
of $G(\Po_n,r)$ with $\dist(\varphi,\psi) < 2r$. 
Hence by a further application of the multivariate Mecke formula,
it is precisely equal to $2J_{1,n}$, where $J_{1,n}$
was defined at (\ref{e:J1def}). Therefore by Lemma \ref{l:J1},
the second term in the right hand side of (\ref{0729e}), multiplied
by $n^2$, is  $O(e^{-cnr^d} I_{n,k})$ for some $c>0$.

We turn to the last term in the right hand side of (\ref{0729e}).
For large enough $n$, by Lemma \ref{l:bdy_est} we have
$p_\by \geq f_0 (\theta/3)  r^d$ and also
$(1-p_\bx)^{-k}(1-p_\by)^{-k} \leq 2$
for all $\bx,\by \in A$. Then
\begin{align*}
	&	n^{2 + 2(k-1)}
	\int_{A^k} \int_{A^k \setminus (A \setminus B_{2r}(\bx))^k} 
	h^*_r(\bx) h^*_r(\by)
	(1-p_\bx)^{n-k}(1-p_\by)^{n-k} \nu^k(d\by) \nu^k(d \bx) 
	\\
	& \leq c n^{2k} r^{dk} 
	\int_{A^k} h^*_r(\bx) e^{-np_\bx}
	e^{-  (f_0 \theta/3) n r^d}\nu^k(d\bx) 
\\
	& = I_{n,k} \times O((nr^d)^{k} e^{-f_0  ( \theta /3)nr^d}), 
\end{align*}
which is $O(I_{n,k} \exp(- (f_0 \theta/4) nr^d))$.
Thus there exists $c>0$ such that all terms in the right hand side
of (\ref{0729e}), multiplied by $n^2$, are
$O( e^{-c nr^d} I_{n,k})$. Therefore the right hand side of (\ref{0729b})
is $O( e^{-c nr^d} I_{n,k})$, as required.
\end{proof}

\subsection{LLNs and concentration results for $S_{n,k}$ and $S'_{n,k}$}

A sequence of random variables $(\xi_n)_{n\in \N }$ is said to
{\em converge  completely} to a constant $C$ (written
$\xi_n \tocc C$) as $n \to \infty$ if $\PP[|\xi_n-C|>\ep]$ is summable in $n$
for any $\ep>0$. In particular, 
complete convergence implies almost sure convergence for any sequence  of random variables defined on the same probability space.  We
now prove $S_{n,k}/I_{n,k} \tocc 1$ and 
$S'_{n,k}/I_{n,k} \tocc 1$ under an extra condition on $r=r_n$, namely
\bea
\limsup_{n \to \infty} ( \theta n r_n^d/(\log n)) < 1/(2 f_0).
\label{e:xsublog}
\eea
Using our results about $I_{n,k}$, we can then give a sequence
of constants that are almost surely asymptotic 
to $S_{n,k}$ (in the uniform case) or to $\log(S_{n,k}/n)$ (in the
general case).

\begin{theorem}[Concentration results for $S_{n,k}$ and $S'_{n,k}$]
        \label{t:non-uniform}
        Suppose that $r =r_n$ satisfies 
	(\ref{e:xsublog}) as well as (\ref{e:supcri3}). 
	Given $\eps >0$, there exist $\delta, n_1 \in (0,\infty)$
        such that for all $n \geq n_1$,
	\begin{align}
	\Pr [ |(S_{n,k}/I_{n,k}) -1| > \eps] \leq \exp(-n^{\delta});
		\label{e:concB}
		\\
	\Pr [ |(S'_{n,k}/I_{n,k}) -1| > \eps] \leq \exp(-n^{\delta}).
		\label{e:concP}
		\end{align}
        In particular, as $n \to \infty$ we have the complete convergence
	$ (S_{n,k}/I_{n,k}) \tocc 1, $ and also
        \bea
	(nr^d)^{-1} \log (S_{n,k} /n) \tocc - \theta f_0,
        \label{e:LLNlogK}
        \eea
	and likewise for $S'_{n,k}$ (as $n \to \infty$ through $\N$).
	Finally, in the uniform case,
	\bea
	n^{-1} e^{ f_0 \theta n r^d } (nr^d)^{(k-1)(d-1)} S_{n,k}
	\tocc k^{-1} f_0^{(1-k)(d-1)} \alpha_k, 
	\label{e:Snccunif}
	\eea
	and likewise for $S'_{n,k}$.
\end{theorem}

\begin{remark}
	If (\ref{e:supcri3}) holds but (\ref{e:xsublog}) does not,
	then as mentioned in Section \ref{secintro}, we
	still have $S_{n,k}/I_{n,k} \toP 1$, and (\ref{e:LLNlogK})
	and (\ref{e:Snccunif})
	with convergence in probability.

When $f_0=f_1$, the upper bound in (\ref{e:supcri3}) is $f_0^{-1} \min(2/d,1)$,
so (\ref{e:xsublog}) is a stronger condition than the upper bound in
	(\ref{e:supcri3}) when $d \leq 3$ (but not when $d \geq 4$).
\end{remark}

\begin{proof}[\it Proof of Theorem \ref{t:non-uniform}.]
By definition
	$\EE[S'_{n,k}] =  I_{n,k}$. 
Hence by Lemma  \ref{l:Ilower},
	 given any $\ep>0$ and $f_0^+> f_0$, for $n$ large enough, 
\begin{align*}
	\PP[ |S'_{n,k} - I_{n,k}| \ge \ep I_{n,k}] 
	\leq \Pr[|S'_{n,k}- \EE[S'_{n,k}]|
	\geq n e^{-  f_0^+ \theta nr^d}].
\end{align*}

	Partition $\RR^d$ into cubes of side length $r$. Then the number of 
	$k$-clusters intersecting a fixed cube in the partition is bounded by a constant $c'$ depending only on $d$. Therefore, removing all the Poisson points in any fixed cube of the partition reduces or increases the total number of
	$k$-clusters by at most a constant depending only on $d$.  

	Enumerate the cubes in the partition that intersect
	$A$ as  $C_1,...,C_{m_n}$, where $m_n = O(r^{-d})$ as
	$n \to \infty$.
	For each $\ell\in [m_n]$,  let $\mathcal F_\ell$ be the sigma-algebra generated by $\Po_n\cap (\cup_{i=1}^\ell C_i)$. Set $D_i = \EE[S'_{n,k}|\mathcal F_i]-\EE[S'_{n,k}|\mathcal F_{i-1}]$ for $i\in[m_n]$.
	Then $S'_{n,k}$ is written as the sum of martingale differences
	$S'_{n,k}-\EE[S'_{n,k}]=\sum_{i=1}^{m_n} D_i$. To see why $|D_i|$ is bounded from above by a finite constant, we re-sample the Poisson points in $C_i$ from an 
	independent copy $\Po'_n$ defined on the same probability space. By the
	superposition theorem \cite{LP18},  
	recalling that $K_{k,r}(\X)$ is the number of $k$-clusters
	of  $G(\X,r)$, we have that
\begin{align*}
	D_i = \EE[K_{k,r}(\Po_n)  - K_{k,r}((\Po'_n \cap C_i) \cup
	(\Po_n\setminus C_i))  |\mathcal F_i].
\end{align*}
Hence $|D_i|\le 2c'$ uniformly in $i\in[m_n]$, as discussed in the last
	paragraph.  From here, it follows easily by an application of
	Azuma's inequality \cite[page 33]{Pen03} that  for all $n$ large
\begin{align*}
	\PP[ |S'_{n,k} - \EE[S'_{n,k}]|\ge
	  n e^{-nf_0^+\theta r^d} ] 
	&\le 2\exp \left( -   \frac{ n^2  e^{-2 nf_0^+\theta r^d}
	}{
		8 (c')^{2} m_n   } \right).
\end{align*} 
	By the assumption (\ref{e:xsublog}),
	we may assume $f_0^+$ is chosen so that
	$\limsup_{n \to \infty} n \theta f_0^+ r^d/(\log n) < 1/2$.
	Using this we can find $\de >0$ and $c>0$
	such that for $n$ large,
	the last bound is at most $2 \exp(-c'' n^\delta)$.
	This gives us the result (\ref{e:concP}).

Now consider $S_{n,k}$ (i.e.  the binomial case).
	By Proposition \ref{p:Snmeanbin} we have  for some $c>0$ that
	$
	\EE[S_{n,k}]= I_{n,k} (1+O(e^{-c n r^d})),$
and hence for $n$ large,
\begin{align*}
	\PP[ |S_{n,k} - I_{n,k}| \ge \ep I_{n,k}] & \le
	\PP[ |S_{n,k}- \EE[S_{n,k}]| \ge  \ep I_{n,k} /2]
	\\
	& \leq \Pr[|S_{n,k}- \EE[S_{n,k}]|
	\geq n e^{- n f_0^+ \theta r^d}].
\end{align*}
	  Then (\ref{e:concB}) 
	 can be proved in the same manner as (\ref{e:concP})
	 with the filtration now given
	 by $\cF_\ell = \sigma(X_i: i\le \ell)$ and $D_i = 
	\EE[K_{k,r}(\cX_n) - K_{k,r}(\cX_{n+1}\setminus\{X_i\})|\cF_i]$.
	We omit the details. 

	It is immediate from (\ref{e:concB}) and (\ref{e:concP}) that
	$ (S_{n,k}/I_{n,k}) \tocc 1,  $
	and
	$ (S'_{n,k}/I_{n,k}) \tocc 1.  $

	For (\ref{e:LLNlogK}), note that 
	since $(S_{n,k}/I_{n,k}) \tocc 1$ and $nr^d \to \infty$, 
	$(nr^d)^{-1} \log (S_{n,k}/I_{n,k}) \tocc 0$,
	and thus by Theorem \ref{t:Inonunif},
	$$
	(n r^d)^{-1} \log (S_{n,k}/n) = (nr^d)^{-1} \log(S_{n,k}/I_{n,k}) +
	(n r^d)^{-1} \log (I_{n,k}/n) \tocc - \theta f_0,
	$$
	which is (\ref{e:LLNlogK}). We obtain a similar result for
	$S'_{n,k} $ using the fact that $(S'_{n,k}/I_{n,k}) \tocc 1$.

	In the uniform case,
	we obtain (\ref{e:Snccunif}) using (\ref{e:Ilim}). 
\end{proof}
\section{Asymptotic distribution of the $k$-cluster count}
To present Poisson and normal approximation results for $S_{n,k}$
and $S'_{n,k}$, we recall three notions of distance
which measure the proximity between the distributions
of real-valued random variables $X,Y$. 
The Kolmogorov distance is defined by
\begin{align*}
\dk(X,Y) := \sup_{z\in\RR} |\PP[X\le z] - \PP[Y\le z]|
\end{align*}
The Wasserstein-1 distance is defined by
\begin{align*}
\dw(X,Y):= \sup_{h\in\mathrm{Lip}_1} |\EE[h(X)] - \EE[h(Y)]|,
\end{align*}
where $\mathrm{Lip}_1$ denotes the family of Lipschitz continuous functions with Lipschitz constant at most one. The total variation distance is defined by
\begin{align*}
\dtv(X,Y) := \sup_{A\in \mathcal B(\RR)} |\PP[X\in A] - \PP[Y\in A]|,
\end{align*}
where the supremum is taken over all Borel measurable subsets of $\RR$.
If $X,Y $ both take values in $\Z$, then  this is the
same as $(1/2)\sum_{n \in \Z}| \Pr[X = n]- \Pr[Y =n]|$. 
Convergence in any of the three 
aforementioned distances implies convergence in distribution.

In this section we prove the following result, concerning the asymptotic
distributions of $S_{n,k}$ (the number of $k$-clusters in $G(\cX_n,r)$) and
	 $S'_{n,k}$  (the number of $k$-clusters in $G(\Po_n, r)$), 
	in the mildly dense regime.
	Recall from (\ref{e:exp_Snk})
	that $\E[S'_{n,k}]= I_{n,k}$.
	Recall that $Z_t$ denotes a Poisson variable with mean $t$,
	and let $\cN$ denote a standard normal random variable.
\begin{theorem}
	[Distributional results on the $k$-cluster count]
\label{t:int_k} 
	Suppose that $r = r(n)$ satisfies \eqref{e:supcri3}, and 
set $b:= \limsup_{n \to \infty} (n \theta r^d/\log n) $.  Let $k\in\NN$.
	There exists a finite $c>0$ such that for  all large enough $n$,
\begin{align}
	\dtv(S'_{n,k}, Z_{I_{n,k}})
	\le e^{- c n r^d},
	\label{e:k_pois}
\end{align}
	and if $d \geq 2$ then for all large enough $n$,
\begin{align}
	\dtv(S_{n,k}, Z_{\E[S_{n,k}]}) \le e^{- c n r^d},
	\label{e:k_bin}
\end{align}
	Moreover (for all $d$), given $\eps >0$ we have 
\begin{align}
\label{e:CLT_dk}
	\dk( \Var[S'_{n,k}]^{-1/2}(S'_{n,k}-I_{n,k}), \cN) 
	=O(n^{\eps +(bf_0  - 1)/2}); 
\end{align}
and 
	\begin{align}
\dw( \Var[S_{n,k}]^{-1/2}(S_{n,k}-\EE[S_{n,k}]), \cN)
		= O(n^{\eps + (3bf_0 - 1)/2}),
	\label{e:CLTBi}
\end{align}
	Also if $ d \geq 2$ then there exists $c >0$ such that
	for all large enough $n$,
\begin{align}
	\dk( \Var[S_{n,k}]^{-1/2}(S_{n,k}- \E[S_{n,k}]), \cN)
	\leq e^{-cnr^d},
	\label{e:dKs}
\end{align}
	and a similar result holds for $S'_{n,k}$ for all $d \geq 1$.
\end{theorem}

\begin{remark}
\begin{enumerate}
%
\item
		The normal approximation (\ref{e:dKs})
		(and the corresponding result for $S'_{n,k}$) will be proved 
	via Poisson approximation, while  
		(\ref{e:CLT_dk}) and 
	(\ref{e:CLTBi}) will be obtained by approximating 
		$S_{n,k}$ or $S'_{n,k}$ by directly by a normal random variable.
		It is of interest to compare the rates of normal convergence
		in these results with one another.

When $b=0$, the rates are better in (\ref{e:CLT_dk}) and (\ref{e:CLTBi}) than
		in (\ref{e:dKs}) since $e^{-cnr^d}$ decays more slowly than
		$n^{-\eps}$ for any $\eps >0$.
	To fully compare  when $b>0$, we would need to optimize the value of
		$c$ in (\ref{e:dKs}), which is beyond the scope of this paper.

	\item  
		We conjecture that the rates in (\ref{e:CLT_dk}), and in 
(\ref{e:CLTBi}) for $b=0$, are optimal up to a factor of $n^\eps$.
		When $b>0$ the rate of convergence given
		for $S_{n,k}$ in (\ref{e:CLTBi})
		is worse than the rate given for $S'_{n,k}$ in
		(\ref{e:CLT_dk}). We expect that in this
		case the rate in (\ref{e:CLTBi}) is sub-optimal.

	\item
		When $bf_0 \geq 1/3$, (\ref{e:CLTBi}) does not
		give a CLT for $S_{n,k}$ at all. However,
		provided $d \geq 2$ and (\ref{e:supcri3})
		holds (which amounts to $bf_0 < 2/d$ when $f_0=f_1$), 
		we do still have a CLT for $S_{n,k}$
		(possibly with a sub-optimal rate of convergence) by
		(\ref{e:dKs}).

	\item
		In (\ref{e:CLTBi}) we found it more convenient
		to use use $\dw$ rather than $\dk$. The proof
		of (\ref{e:CLTBi}) is based on a general
		result on normal approximation (in $\dw$)
		for stabilizing functionals of binomial point processes
		(Lemma \ref{l:fromChat}) which is itself based on a result in
		\cite{Cha08}. It might be possible to obtain a similar
		result with  $\dk$ instead of $\dw$
		by utilizing \cite{LrP17} instead of \cite{Cha08}.

	\item
		A result along the lines of (\ref{e:CLT_dk})
		(without any rate of convergence)
		is proved in \cite{Pen18} for a class of {\em soft}
		RGGs where the probability of two vertices
		being connected given their locations at $x,y$ say,
		is some function of $x$ and $y$ (called the
		{\em connection function}), rather than being
		$\1\{\|y-x\| \leq r\}$ as in the graphs considered here.
		However, the result in \cite{Pen18} requires
		the connection function to be bounded away from 1,
		so the result there does not cover the RGGs that
		we consider here.
\end{enumerate}
\end{remark}

\subsection{Poisson approximation for $S'_{n,k}$}
\label{s:PoSprime}

Let $\bN(\R^d)$ be the space of all finite subsets of $\R^d$, equipped with
the smallest $\sigma$-algebra ${\cal S}(\R^d)$ containing the
sets $\{\cX \in \bN(\R^d): |\cX \cap B|= m\}$ for all Borel $B \subset \R^d$
and all $m \in \N \cup \{0\}$. Given $m \in \N$, let $\bN_m(\R^d):= \{\X \in
\bN(\R^d):|\X|=m\}$.

Our main tool for proving the Poisson approximation result (\ref{e:k_pois}) in
Theorem \ref{t:int_k} is the following coupling bound in \cite{Pen18} adapted to our situation (i.e. without marking). The function  $g$ in the next result has
nothing to do with the $g$ in Section \ref{s:Boolean}.

\begin{lemma}[{\cite[Theorem 3.1]{Pen18}}]
\label{t:JAP}
	Let $g: \bN_k(\R^d)
	\times \mathbf{N}(\RR^d)\to \{0,1\}$  be measurable.
	Define  
\begin{align*}
W:=F(\Po_n) := \sum_{\psi\subset \Po_n, |\psi|=k} g(\psi, \Po_n\setminus \psi).
\end{align*}
Let $n >0, k \in \N$.
	For $\bx=(x_1,...,x_k)\in(\RR^d)^k$ with distinct entries,
	set $p(\bx):= 
	\EE[g(\{x_1,\ldots,x_k\},\Po_n)]$ and set $\mu = n \nu$.
	Assume that for $\mu^k$-a.e. $\bx$ with $p(\bx)>0$, we can find coupled random variables $U_\bx, V_\bx$ such that 
\begin{itemize}
\item $\mathscr L(U_\bx) = \mathscr L(W)$;
\item $\mathscr L(1+ V_\bx) = \mathscr L(F(\Po_n\cup\{x_1,...,x_k\})| g
	(\{x_1,\ldots,x_k\},\Po_n)=1)$.
\end{itemize}
Then 
\begin{align}
	\dtv(W, Z_{\EE[W]}
	) \le \frac{\min(1, \EE[W]^{-1})}{k!}\int \EE[|U_\bx-V_\bx|] p(\bx)\mu^k(d\bx).
	\label{0730b}
\end{align}
\end{lemma}

We restate the first Poisson approximation result (\ref{e:k_pois}) 
from Theorem
\ref{t:int_k} in the following proposition. 
Recall that $Z_t$ denotes a Poisson variable with mean $t$.

\begin{proposition}
\label{p:poi}
	Let $k \in \N$. Then
there exist finite constants $c, n_0 >0$ such that
$$
		\dtv(S'_{n,k}, Z_{I_{n,k}})  \leq e^{-c n r^d},
	~~~~~ \forall ~ n \geq n_0.
	$$
\end{proposition}

\begin{proof}
In view of \eqref{e:def_Snk}, we apply Lemma \ref{t:JAP} with 
	$$
	g(\psi,\varphi):=
	h_r(\psi)\1\{\varphi \cap B_r(\psi)= \emptyset\}.
	$$
	For $\bx  \in (\R^d)^k$,
	with $h_r(\bx)=1$,
	we construct  coupled random variables $(U_\bx,V_\bx)$ as follows.  
	Define $U_\bx := \sum_{\varphi \subset \Po_n, |\varphi | =k} g(\varphi, \Po_n \setminus \varphi)$, and
\begin{align*}
V_\bx := \sum_{\ph\subset \Po_n\setminus 
	B_r(\bx), |\ph|=k} g(\ph, \Po_n\setminus B_r(\bx)
	\setminus \ph).
\end{align*}
This coupling satisfies the distributional 
	requirement because the conditional distribution of 
	$\Po_n$ given the event $\{ g(\bx,\Po_n)=1\}$ is the same as the distribution of 
	$\Po_n\setminus B_r(\bx) 
	$.

There are two sources of contribution to the change $U_\bx-V_\bx$ of $k$-cluster counts after removing all the Poisson points in $B_r(\bx)$.
First, after removal, all $k$-clusters of $G(\Po_n,r)$
that were originally intersecting $B_r(\bx)$ are destroyed, therefore 
reducing the $k$-cluster count.  Second, every $k$-set
	$\ph\subset \Po_n\setminus B_r(\bx)$ satisfying the two properties
\begin{itemize}
	\item[(a)] $g(\ph, \Po_n\setminus 
		B_r(\bx)
		\setminus \ph)=1$;
	\item[(b)] $\Po_n (
		B_r(\bx)
		\cap  B_r(\ph) )\ge 1$;
\end{itemize}
becomes a $k$-cluster \emph{only} after removing all the Poisson points
	in $ B_r(\bx)$, thereby increasing the number of $k$-clusters.  
		Let $\xi_1(\bx)$ denote the number of $k$-clusters
		of $G(\Po_n,r)$ that
		intersect $B_r(\bx)$ and let $\xi_2(\bx)$ denote the
		number of $k$-subsets $\ph$ of
		$\Po_n\setminus  B_r(\bx)$ satisfying property (a)
		and property (c) $B_r(\ph)\cap B_r(\bx)
		\neq\emptyset$. It is clear that (b) implies (c) and 
\begin{align}
|U_\bx-V_\bx|\le \xi_1(\bx)+\xi_2(\bx).
	\label{e:UVX1X2}
\end{align}

We estimate $\EE[\xi_1(\bx)]$ and $\EE[\xi_2(\bx)]$ separately. Since
$$
\xi_1(\bx) = \sum_{\ph \subset \Po_n: |\ph|=k, \ph \cap B_r(\bx) \neq \emptyset}
h_r(\varphi)
\1\{ (\Po_n \setminus \ph)  \cap B_r(\ph) = \emptyset \},
$$
applying the multivariate Mecke equation leads to 
$$
\E [ \xi_1(\bx) ] = \frac{1}{k!} \int_{A^k \setminus (A \setminus B_r(\bx))^k}
h_r (\by) e^{-n \nu(B_r(\by))} (n \nu)^k(d\by).
$$
If $\by \in A^k \setminus (A \setminus B_r(\bx))^k$ and $h_r(\by)=1$ then
$\by \subset B_{kr}(\bx)^k$. Moreover by Lemma \ref{l:bdy_est},
if $n$ is large enough then $\nu(B_r(\by)) \geq f_0 (\theta/3) r^d$
for any $\by \in A^k$. Hence for all $n$ large enough and all $\bx$,
$$
\E[ \xi_1(\bx)] \leq   \frac{n^k}{k!} (k \theta (kr)^d \fmax)^k e^{- f_0 (\theta/3) n r^d}.
$$
Therefore setting $p((x_1,\ldots,x_k))= \EE[g(\{x_1,\ldots,x_k\},\Po_n)]$, and using (\ref{e:exp_Snk}),
we have
\begin{align}
	\int_{A^k} \EE[ \xi_1(\bx)] p(\bx) (n \nu)^k(d\bx)
	& \leq (\fmax \theta k^{d+1})^k(nr^d)^k e^{-(\theta/3)f_0nr^d} 
	(\frac{n^k}{k!}) \int_{A^k} p(\bx) \nu^k(d\bx)
\nonumber \\
	& = 
	 (\fmax \theta k^{d+1})^k (nr^d)^k e^{-(\theta/3)f_0nr^d} 
	I_{n,k}.
	\label{0730c}
\end{align}

 Set $\gamma(\bx,\by)= \1\{ r < \dist(\bx,\by) \leq 2r\}$.
 By the multivariate Mecke equation
\begin{align*}
\EE[\xi_2(\bx)] &= \frac{n^k}{k!}
	\int_{A^k}
	h_r(\by)\gamma(\bx,\by) 
	e^{-n \nu(B_r(\by) \setminus B_r(\bx))}
	\nu^k (  d\by),
\end{align*}
and therefore writing $B_r(\bx,\by)$ for $B_r(\bx) \cup B_r(\by)$,
we have that
\begin{align*}
	\int_{A^k} \EE[\xi_2(\bx)] p(\bx)(n \nu)^k (d\bx)
	= \frac{n^{2k}}{k!} \int_{A^k} \int_{A^k}
	h_r(\bx) h_r(\by) \gamma(\bx,\by)
	e^{-n \nu(B_r(\bx,\by))}
	\nu^k(d \by) \nu^k(d\bx).
\end{align*}
By (\ref{e:J1def}) this expression is equal to $k! J_{1,n}$, and therefore
by Lemma \ref{l:J1}
it is $O(e^{-c'nr^d}I_{n,k})$ for some $c' >0$.

Combining this with (\ref{0730c}), and using (\ref{e:UVX1X2}), 
we obtain for suitable $c>0$ that
$$
\int_{A^k} |U_{\bx} - V_{\bx}| p(\bx) (n \nu)^k(d\bx)
= O( e^{- c  n r^d}  I_{n,k}).  
	$$
Applying Lemma \ref{t:JAP} with the present choice of $g$
(so that the $W$ of that result is $S'_{n,k}$)
gives the desired bound in Poisson approximation, completing the proof. 
\end{proof}

\subsection{Poisson approximation for $S_{n,k}$}

For Poisson approximation in the binomial setting, i.e. for $S_{n,k}$,
we use the following result from 
\cite[Theorem II.24.3]{Lin92}
or
\cite[Theorem 1.B]{BHJ92}.
\begin{lemma}
	\label{l:Lindvall}
	Let $n \in \N$.
	Suppose $Y_1,\ldots,Y_n$ are  Bernoulli random variables
	on a common probability space.
	Set $W:= \sum_{i=1}^n Y_i$,
	and  $p_i := \E[Y_i]$ for $i \in [ n]$.
	Suppose  for each $i \in [n]$ that 
	there exist coupled random variables $U_i,V_i$ such that
	 $\LL(U_i)= \LL(W)$ 
	 and $\LL(1+V_i) = \LL(W|Y_i=1)$.
	 Then
	 $$
	 \dtv(W, Z_{\EE[W]} ) \leq  (\min(1,1/\EE[W]))\sum_{i=1}^n p_i
	 \EE[ |U_i- V_i|].
	 $$
\end{lemma}
We use this to obtain the analogue of Proposition \ref{p:poi} in
the binomial setting, i.e. the second Poisson approximation result
(\ref{e:k_bin})
in Theorem \ref{t:int_k}: 
\begin{proposition}
	\label{p:PoApproxBi}
	Suppose $d \geq 2$.  Let $k \in \N$. 
	There exist $c >0$, $n_0 \in (0,\infty)$ such that 
$$
	\dtv(S_{n,k},Z_{\EE[S_{n,k}]})
	\le e^{- c n r^d},
	 ~~~~~~ \forall ~ n \geq n_0.
	$$
\end{proposition}
We shall prove this in stages.
	To apply Lemma \ref{l:Lindvall}
	to $S_{n,k}$, we let $Y_i$ be the indicator of the
event that $|\cC_r(X_i,\cX_n)| =k$ {\em and} $X_i$ is the left-most
point of $\cC_r(X_i,\cX_n)$ (we called this $\xi_i$ at (\ref{e:xidef})).
Then $S_{n,k}= \sum_{i=1}^n Y_i $.

We need to define $U_i$,  $V_i$ for each $i \in [n]$ so that
$\LL(U_i) = \LL(S_{n,k})$, and $\LL(1+V_i)= \LL(S_{n,k}|Y_i=1)$,
and so that we can find a good bound for $\EE[|U_i-V_i|]$.
We do this for $i=1$ as follows. First define the event
$$
\cE:= \{Y_1=1\} \cap \{\cC_r(X_1,\X_n) = \{X_1,\ldots,X_k\}\}.
$$
Let $(\tX_1,\ldots,\tX_k) $ be a random vector in $(\R^d)^k$
with $\LL(\tX_1,\ldots,\tX_k) = \LL((X_1,\ldots,X_k)|\cE)$.
Also let $(X_{i,j}, i \in [n], j \in \N)$ be 
an array of independent $\nu$-distributed random variables,
independent of $(\tX_1,\ldots,\tX_k)$.
Set $\X_{n,1} = \{X_{1,1},\ldots,X_{n,1}\}$.

For $k < i \leq n$ set $J_i = \min \{j: X_{i,j} \notin \cup_{\ell=1}^k
B_r(\tX_\ell) \}$ and set $\tX_{i} := X_{i,J_i}$.
Then set $\X_{n,2}:= \{\tX_1,\ldots,\tX_n\}$.

In other words, we sample the random vector $(\tX_1,\ldots,\tX_k)$
from the conditional distribution of $(X_1,\ldots,X_k)$ given
that $\cE$ occurs, independently of $\X_{n,1}$.
Given the outcome of $\tX_1,\ldots,\tX_k$,
	for $i \in [n] \setminus [k]$, if $X_{i,1} \notin
	\cup_{\ell=1}^k B_r(\tX_\ell)$
we take $\tX_i = X_{i,1}$. Otherwise
 we re-sample a random  vector with distribution $\nu$ repeatedly
until we get a value that is not in $\cup_{\ell=1}^k B_r(\tX_\ell)$,
and call this $\tX_i$.
Thus, given the value of $(\tX_1,\dots,\tX_k)$,
the distribution  of $\tX_i$ is given by the measure
$\nu $ restricted to $A \setminus \cup_{\ell=1}^k B_r(\tX_\ell)$,
normalised to a probability measure. 

Set $U_1 := K_{k,r}(\cX_{n,1})$, and
$V_1 := K_{k,r}(\X_{n,2}) -1$.
Note that in this coupling  we make no attempt to make  a `good'
coupling of $(\tX_1,\ldots,\tX_k)$
to $(X_{1,1},\ldots,X_{k,1})$. However, for
$k+1 \leq i \leq n$ the variables $\tX_{i}$,
are the same as $X_{i,1}$ except in
those rare cases where $X_{i,1} $ lies in
$B_r((\tX_1,\ldots,\tX_k))$. This is what makes
the coupling effective.
Clearly $\LL(U_1)= \LL(S_{n,k})$.

\begin{lemma}
	\label{l:Vdist} With $V_1$ as just defined,
	$\LL(1+V_1)= \LL(S_{n,k}|Y_1=1)$. 
\end{lemma}
	\begin{proof}
		Given $Y_1=1$,
we have that $\cC_r(X_1,\cX_n)$ has $k$ vertices with $X_1$
the left-most of these, and by exchangeability we
can assume the other vertices of $\cC_r(X_1,\X_n)$
are $X_2,\ldots,X_k$
without affecting the distribution of the point process $\cX_n$.
In other words, we have
$
\LL(\X_n|Y_1=1)= \LL(\X_n| \cE
		).$ Moreover, we claim that 
		\begin{align}
		\LL((\tX_1,\ldots,\tX_n)) = \LL((X_1,\ldots,X_n)|\cE). 
			\label{e:condit}
		\end{align}
		This implies that
$
\LL(\{\tX_1,\ldots,\tX_n\}) = \LL(\X_n|\cE) = \LL(\X_n|Y_1=1),
$
		and hence $\LL(1+V_1) = \LL(K_{k,r}(\{\tX_1,\ldots,\tX_n\}))
= \LL( S_{n,k}|Y_1=1)$, as required.

		It remains to confirm the claim (\ref{e:condit}).
		The reader may think this is obvious;
		we provide a sketch proof.
For $\bx = (x_1,\ldots,x_n) \in (\R^d)^n$, let
us 
set $\bx_1^k := (x_1,\ldots,x_k)$, $\bx_{k+1}^n :=
(x_{k+1},\ldots,x_n)$ and $g_r(\bx_1^k,\bx_{k+1}^n):=
\prod_{\ell=k+1}^n \1\{x_\ell \notin  B_r(\bx_1^k)\}$.
		Then with $h_r^*$ defined just after (\ref{e:Brhrdef}),
\begin{align}
	\Pr[(X_1,\ldots,X_n) \in d\bx; \cE ] 
	=  h_r^*(\bx_1^k) g_r(\bx_1^k,\bx_{k+1}^n) 
	\nu^n(d\bx),
	\label{e:0902a}
\end{align}
so that $\Pr(\cE) = 
\int_{A^n} h_r^*(\bx_1^k) g_r( \bx_1^k,\bx_{k+1}^n)
		\nu^n(d \bx)$. For $\bx_{1}^k \in (\R^d)^k$,  
		let us set $I(\bx_1^k) : = \int_{(\R^d)^{n-k}}
		g_r(\bx_1^k, \bx_{k+1}^n) \nu^{n-k} (d\bx_{k+1}^n).$
		Then by (\ref{e:0902a}),
\begin{align*}
	\Pr[(\tX_1,\ldots,\tX_k) \in d\bx_1^k] & =
	\Pr[(X_1,\ldots,X_k) \in d\bx_1^k | \cE]
	= \frac{ h^*_r(\bx_1^k) 
	 I(\bx_1^k)
	 \nu^{k}
	(d\bx_1^k) }{\int_{A^n}
	h_r^*(\bx_1^k)
	g_r(\bx_1^k,\bx_{k+1}^n)
	\nu^n(d\bx) }.
\end{align*}
Then
\begin{align*}
	\Pr[(\tX_1,\ldots,\tX_n) \in d\bx] &=
	\frac{g_r(\bx_1^k,\bx_{k+1}^n) \nu^{n-k}(d\bx_{k+1}^n) }{
		I(\bx_1^k)	}
	\times \Pr[ (\tX_1,\ldots \tX_k) \in d\bx_1^k]
	\\
	& = \frac{ h^*_r(\bx_1^k) g_r(\bx_1^k,\bx_{k+1}^n)
	\nu^n(d(\bx_1^k,\bx_{k+1}^n))}
	{\int_{A^n} h_r^*(\bx_1^k) g_r(\bx_1^k,\bx_{k+1}^n) \nu^n(d \bx)}
	\\
	& = 
	\Pr[(X_1,\ldots,X_n) \in d\bx| \cE ], 
\end{align*}
where the last line comes from (\ref{e:0902a}).
Thus we have (\ref{e:condit}).
\end{proof}

\begin{proof}[Proof of Proposition \ref{p:PoApproxBi}]
	We need to find a useful bound on $\E[|U_1-V_1|]$.
	Note that the `new' point process $\mathcal X_{n,2}=\{\tX_1,\ldots\tX_n\}$
	is obtained from the `old' point process
	$\X_{n,1}$ by replacing  the points
	$X_{1,1},\ldots,X_{k,1}$ with $\tX_1,\ldots,\tX_k$ 
	and also, for those $i \in [n] \setminus [k]$ such
	that $J_i >1$, replacing the point $X_{i,1} $ with $\tX_i$,
	leaving the other points unchanged.

	We refer to vertices and components (here also called {\em clusters})
	of $G(\X_{n,1},r)$ as being {\em old} while vertices  and clusters of
	$G(\cX_{n,2},r)$ are {\em new}.
	We write $\tcC$ for $\{\tX_1,\ldots,\tX_k\}$. By construction
	$\tcC$ is a cluster of $G(\cX_{n,2},r)$. 

The value of $|U_1-V_1|$ is bounded by the sum of the following variables
$N_i, 1 \leq i \leq 6$.

	$N_{1}$ is the number of old $k$-clusters involving 
	$X_{1,1},\ldots,X_{k,1}$,
	and $N_2$ is the number of new $k$-clusters within distance $r$
of one of $X_{1,1},\ldots,X_{k,1}$ but not using any of the
new 	vertices  $\tX_i$ with $i>k, J_i>1$, and with no
	vertex in $B_{2r}(\tcC)$.
These will be affected by removing $X_{i,1}$
	for $1 \leq i \leq k$.

$N_3$ is the number of old $k$-clusters intersecting $B_r(\tX_i)$
	for some $i  \in [n] \setminus [k]$ with $J_i > 1$, and not using
	$X_{i,1}$ (if such a cluster does use $X_{i,1}$ it is included in
	$N_5$ below).
	$N_4$ is the number of new $k$-clusters involving $\tX_i$
	for some $i >k$ with $J_i >1$, and having no vertex in $B_{2r}(\tcC)$.
	These are affected by the creation of new vertices at $\tX_i$
	with $i >k, J_i>1$.

	$N_5$ is the number of old $k$-clusters having at least one 
	vertex in $B_r(\tcC)$.
These clusters are affected by the
	removal of old vertices in $B_r(\tcC)$.

$N_6$ is the number of new $k$-clusters having at least one
	vertex in $B_{2r}(\tcC)$,
other than $\tcC$ 
itself. These could be created from previously larger clusters
	due to the removal of old vertices in 
	$B_r(\tcC)$.

	We estimate $\EE[N_i]$ for each $i \in [6]$,
	repeatedly using the fact that $\nu(B_s(x)) \geq (\theta/3)f_0s^d$
	for
	all small enough $s>0$
	and 
	all $x \in A$
	by Lemma \ref{l:bdy_est}.
	We have for large enough $n$ that
	\begin{align*}
		\EE[N_1] & \leq k \binom{n-1}{k-1}
		\int_{A^k} h_r(\bx) (1- \nu(B_r(\bx)))^{n-k} \nu^k(d \bx)
		\\
		& \leq 2 k n^{k-1} (\fmax \theta ((k-1)r)^d)^{k-1}
		\exp(-f_0 (\theta/3) n r^d),
	\end{align*}
	and writing $h_r(x,\bx)$ for $h_r((x,\bx))$, we have
	\begin{align*}
		\EE[N_2] & \leq  k \binom{n-k}{k}
		\int_A \int_{A^{k}} h_r(x,\bx) (1- \nu(B_r(\bx)))^{n-2k} 
		\nu^k(d \bx) \nu(dx)
		\\
		& \leq 2  n^{k} (\fmax \theta (kr)^d)^{k}
		\exp(-f_0 (\theta/3) n r^d).
	\end{align*}
	Using the fact that $\Pr[J_i>1] \leq k \fmax \theta r^d$
	for $i >k$, and the point process
	$\X_{n,1} \setminus \{X_{i,1}\}$ is independent of 
	the event $\{J_i >1\}$ and the random vector $\tX_i$, we obtain that
	\begin{align*}
		\EE[N_3] & \leq (n-k) (k \fmax \theta r^d) \binom{n-1}{k}
		\sup_{x \in A}
		\int_{A^{k}} h_r(x,\bx) (1- \nu(B_r(\bx)))^{n-1-k}
		\nu^{k}(d \bx)
		\\
		& \leq c (nr^d)^{k+1} \exp(- f_0 (\theta/3) n r^d).
	\end{align*}
Using the same bound on $\Pr[J_i >1]$ and the fact that for $i \in [n] 
\setminus [k]$ the  distribution
of $\X_{n,2} \setminus \{\tX_1,\ldots,\tX_k,\tX_i\},$ 
given $(\tX_1,\ldots,\tX_k,\tX_i)$ is that of
a sample of size $n-k-1$ from the restriction of $\nu$ to $A \setminus
\cup_{\ell=1}^k \tX_\ell$ normalized to be a probability measure,
we have that
	\begin{align*}
		\EE[N_4] & \leq (n-k) (k \fmax \theta r^d)
		\binom{n-k-1}{k-1}
		\sup_{x \in A} 
		\left\{ (2 \fmax \theta (kr)^d)^{k-1}
		(1- \nu(B_r(x)))^{n-2k} \right\}
		\\
		& \leq c (n r^d)^k \exp(-f_0 (\theta/3) n r^d).
	\end{align*}
	Next,
	\begin{align*}
		\EE[N_5] & \leq n k \fmax \theta r^d
		\binom{n-1}{k-1} \sup_{x \in A} \int_{A^{k-1}}
		h_r(x,\bx) (1- \nu(B_r(x,\bx)))^{n-k} \nu^{k-1}(d\bx)
		 \\ & \leq c (nr^d)^{k} \exp(- f_0 (\theta/3) nr^d).
	\end{align*}

	We give an estimate for $\E[N_6]$ in Lemma \ref{l:N7} below.
	Combining  this with the earlier estimates for
	$N_1,\ldots,N_5$, we obtain that there exists $\eps >0$ that 
	$$
	\E [ |U_1-V_1|] = O(\exp(-  \eps n r^d)).
	$$
By the exchangeability of $X_1,\ldots,X_n$, we can construct
$U_i, V_i$ similarly for each $i \in [n]$, with the same bound
for $\E[|U_i-V_i|]$. Also we know
$\E[S_{n,k}] \to \infty$ as $n \to \infty$ by Proposition 
\ref{p:Snmeanbin} and Lemma  \ref{l:Ilower}.
Then by Lemma \ref{l:Lindvall} (with the $W$ of that result
equal to our $S_{n,k}$) we obtain for  $n$ large enough  that
\begin{align*}
	\dtv(S_{n,k},Z_{\E[S_{n,k}]} ) & \leq (\E[S_{n,k}])^{-1}
	\sum_{i=1}^n \E[Y_i]
	\times O(\exp(-n \eps
	r^d))
	 = O(\exp(-n \eps
	r^d)),
\end{align*}
as required.
\end{proof}

	\begin{lemma}
		\label{l:N7}
		Suppose $d \geq 2$.
		Let $N_6$ be the number of $k$-clusters of $G(\X_{n,2},r)$
		within distance $2r$ of $\{\tX_1,\ldots,\tX_k\}$,
		other than $\{\tX_1,\ldots,\tX_k\}$ itself. Then
		there exists a constant $c >0$ such
		that $\E[N_6] = O(\exp(-c nr^d))$ as $n \to \infty$.
	\end{lemma}
	It is harder to find a good upper bound for $\EE[N_6]$
	than for $\E[N_i], 1 \leq i \leq 5$.
	The main problem is that for fairly large $k$ (for example, $k=12$
	for $d=2$) there can be a configuration of 
	$\{\tX_1,\ldots,\tX_k\}$ such that the union of balls of
	radius $r$ centred on these points surrounds a very small
	(but non-empty)
	region.  To deal with this, we verify that this eventuality
	is very unlikely because of the compression phenomenon
	mentioned before: given small but fixed $\delta_1 >0$,
	a $k$-cluster is very likely to be
	compressed within a small ball of radius $\delta_1 r$,
	and if this happens the kind of
	possibility just mentioned cannot happen. We can deal
	with the probability  of non-compressed $k$-clusters
	by other means.

	The proof of Lemma \ref{l:N7} uses the following
	geometrical lemma.
	 The lemma fails when $d=1$, and
	 is the reason for the restriction to $d \geq 2$ in the
	 statement of  Proposition \ref{p:PoApproxBi}. It may be
	 possible to prove Lemma \ref{l:N7}, and hence Proposition
	 \ref{p:PoApproxBi}, by other means when $d=1$.
 \begin{lemma}
 \label{l:geom}
Suppose $d \geq 2$.  There exists a constant $\delta_1 >0$ such that
we have for all small enough $s>0$ that
\bea
	 \label{e:ByBx}
	 \vol (B_s(y) \cap A \setminus B_{(1+\delta_1)s}(x))
\geq \delta_1 s^d, ~~ \forall ~x \in A, y \in A \setminus B_s(x).
\eea
\begin{proof}
	Suppose $\delta_1 < 1/4$.
	If $x \in A, y \in A^{(s)} \setminus B_s(x)$ then
	the set
	$B_s(y)  \setminus B_{(1+ \delta_1)s}(x) $ contains a ball of
	radius $s/4$, and is contained in $A$,  so that 
	$\vol (B_s(y) \cap A \setminus B_{(1+\delta_1)s}(x)) \geq \theta (s/4)^d$.

	Let $\delta_2 = 0.01$, and
	 let $\tau \in (0,\tau(A))$.
		We show first that 
	 for all $s \in (0, \delta_2 \tau]$ and all $y \in A \setminus 
	 A^{(s)} $ 
	 we can find a unit vector  $e(y)$ such that
	 \begin{align}
	 B^*(y;s, \delta_2, e(y)) \subset A,
		 \label{e:halfball-}
	 \end{align}
	 where
	the set $B^*(y;s, \eta, e) $ is as defined at
	 \cite[P98]{Pen03}, namely, it is the set of $ z \in B(x,r)$
	 such that $\langle (z-y) , e \rangle > \eta s$, where
	 $\langle \cdot,\cdot \rangle$ is the usual Euclidean inner product.

	 Suppose $0 < s \leq \delta_2 \tau$. 
	 Given $y \in A \setminus A^{(s)}$, let $w $ be the nearest point of $\partial A$
	 to $y$. Let $e= \hat n_w$. 
	 By translating and rotating $A$ we can assume without loss of generality that $y=o$
	 and $e = e_d := (0, \ldots,0,1)$, the $d$th coordinate unit vector.
	 For $u \in \R^{d-1} $ with
        $\|u\| < \tau$, define
        $$
        \phi(u):= \sup \{a \in [-\tau,\tau]: (u,a) \notin A\}.
        $$
	Then as in the proof of \cite[Lemma 3.6]{HPY25}, $\phi(u) \leq \|u\|^2/\tau$.
	But if $z \in B^*(y,s,\delta_2,e)$ then writing $\pi$
	for projection onto the first $d-1$ coordinates,
	since $s \leq \delta_2 \tau$
	we have
	$$
	\langle z, e \rangle \geq \langle z-y,e \rangle \geq \delta_2 s \geq   
	s^2/\tau \geq \phi(\pi(z)) 
	$$
	so that $z \in A$. This gives us 
		 \eqref{e:halfball-} for $s < \delta_2 \tau$.

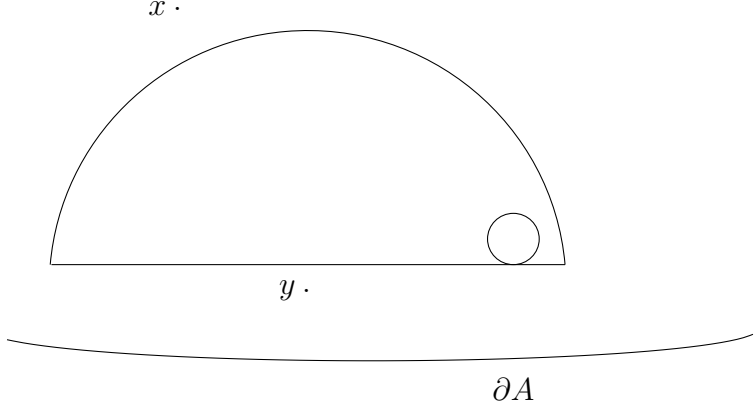
\begin{figure}
	\centering
	\begin{tikzpicture}[scale=1.7]
		
		\draw (0,0) arc (5:175:2cm);
		\draw (0,0) -- (-3.975,0) ;
		\node at (-2,-0.2) {.};
		\node[left] at (-2,-0.2) {$y$};
		\node[left] at (-3,2) {$x$};
		\node at (-3,2) {.};
		
		\draw (-0.4,0.2) circle (0.2);
		
		\draw[domain=200:350,smooth,variable=\x] plot ({3*cos(\x)-1.5},{0.25*sin(\x)-0.5});
		\node[below] at (-0.4,-0.8) {$\partial A$};
	\end{tikzpicture}
	\caption{Illustration for the proof of Lemma \ref{l:geom}.
	The small circle represents $B_{\delta_2 s}(z)$. The larger segment
	represents $B^*(y,s,\delta_2,e(y))$.}
\end{figure}

	 Let $s \in (0,s_0)$.
	 Pick $y \in A \setminus A^{(s)}$ and let $e= e(y)$
	 be as just described.  
	 Without loss of generality, now assume
	 $y=o$ and $e = e_d$.
	 The following argument is
	 illustrated in Figure 3.

	 Let $x \in A \setminus B_{s}(y)$. After a further
	 rotation we can also assume without loss of generality that
	 $x= -ue_1 + ve_d$ for some $u \geq 0$ and $v \in \R$
	 with $u^2+ v^2 \geq s^2$, where $e_1 := (1,0,\ldots,0)$.
	 It follows from this that $\|s e_1 -x\| \geq \sqrt{2} s$.

	 Let $z = (1-7 \delta_2)se_1 + 2 \delta_2s e_d$. By the
	 triangle inequality
	 $\|z-x\| \geq (\sqrt{2} - 9 \delta_2)s \geq (1+ 3 \delta_2) s$.
Hence $B_{\delta_2s}(z) \subset \R^d \setminus B_{(1+ \delta_2)s}(x)$.
Moreover $B_{\delta_2 s}(z)^o \subset B^*(y;s, \delta_2,e_d) \subset A$. 
Therefore
\bean
\vol (B_s(y) \cap A \setminus B_{(1+\delta_2)s}(x))
	\geq \vol(B_{\delta_2 s}(z)^o) =
	\theta (\delta_2 s)^d.
\eean
	Thus taking $\delta_1 = \min(\delta_2,\theta \delta_2^d)$
	we have (\ref{e:ByBx}).
 \end{proof}
	 \end{lemma}
	 \begin{proof}[Proof of Lemma \ref{l:N7}]
Choose  $\delta_1 \in (0,1/2)$ with the property in Lemma \ref{l:geom}.
Let $\bx = (x_1,\ldots,x_k) \in (\R^d)^k$ with $x_1 \prec x_i$
and $\|x_1 - x_i\| \leq \delta_1 r$ for $2 \leq i \leq k$ (here $\prec$
refers to the lexicographic ordering).  For $i \in [n] \setminus [k]$,
let $\X_{i,k,n} := \{\tX_{k+1},\ldots,\tX_n\} \setminus\{\tX_i\}$.  Let
\begin{align*}
N_6(\bx):=  \sum_{i=k+1}^n \1\{\tX_i \in B_{2r}(\bx) \setminus B_r(\bx),
\X_{i,k,n} (B_r(\tX_i) \setminus B_{(1+ \delta_1)r}(x_1) ) < k \}.
\end{align*}
Then using the property from Lemma \ref{l:geom} (taking $y=\tX_i$ and 
 $x = x_1$), and then a  Chernoff-type bound (for example, 
 \cite[Lemma 1.1]{Pen03}), we have for large $n$ that
\begin{align*}
\E[N_6(\bx)] & \leq (n-k)2k  \fmax \theta (2r)^d \Pr[\Bin((n-k-1), f_0 \delta_1 
	r^d) \leq k-1 ]
\\
& \leq  \exp(-  f_0 (\delta_1/2) n r^d).
\end{align*}
Let $\MM$ be the event that
$\{\tX_1,\ldots,\tX_k\} \subset B_{\delta_1 r} (\tX_1)$, i.e.  `compressed'.
If $\MM$ occurs then $N_6 \leq N_6((\tX_1,\ldots,\tX_k))$, and thus we have
	\begin{align}
	\E[N_6 |\MM] \leq \exp(-f_0 (\delta_1/2) nr^d).
		\label{e:N7comp}
	\end{align}
		 Also,
	\begin{align}
		\Pr[\MM^c] & = \Pr[ \cup_{i=2}^k \{X_i \notin B_{\delta_1 r}(X_1) \}
	| \{h_r^*(X_1,\ldots,X_k)=1\} \cap \cap_{\ell =k+1}^n \{X_\ell \notin
	\cup_{i=1}^kB_r(X_i) \} ]
		\nonumber \\
		& = \frac{ \int_{A} \Pr[ E_x] \nu(dx) }{\int_A \Pr[F_x] \nu(dx)}
		\label{e:PrMc}
	\end{align}
	where we set 
	\begin{align*}
		F_x & := \{h_r^*(x,X_2,\ldots,X_k) =1\} \cap  
	\cap _{\ell =k+1}^n \{X_\ell \notin B_r((x,X_2,\ldots, X_k))\};
		\\
		E_x & := \{ 
	\{X_2,\ldots,X_k \} \setminus B_{\delta_1 r}(x) \neq \emptyset
	\} \cap F_x.
	\end{align*}

	Set $b = \limsup_{n \to \infty} ( \theta nr^d/\log n)$
	and note that $b(f_0 -f_1/2) < 1/d$ by (\ref{e:supcri3}).
	Take $\delta_2 >0$ such that 
	\begin{align}
		b(f_0-f_1/2) + b f_0 \delta_2 < 1/d,
		\label{e:de2}
	\end{align}
and such that moreover for all $y \in \R^d \setminus B_{\delta_1}(o)$ we have
$\vol(B_1(y) \setminus B_1(o)) \geq \delta_2 \theta$, and also $\delta_2 < 1/3$.
Then for all small enough $s >0$ and all
	$x \in A^{(3s)}, y \in A \setminus B_{\delta_1s}(x)$ we have
	$\vol(A \cap B_s(y) \setminus B_s(x) ) > \delta_2 \theta s^d$.
Moreover, $h_r^*(x,x_2,\ldots,x_k)=0$ unless $x_2,\ldots,x_k$ all lie in 
 $B_{(k-1)r}(x)$, and hence for all large enough $n$ and all $x \in A^{(3r)}$ we have
	\begin{align*}
		\Pr[ E_x ] & \leq
		\int_{(B_{kr}(x))^{k-1} \setminus (B_{\delta_1r}(x))^{k-1}} 
		(1- \nu(B_r(x,\bx)))^{n-k}
		\nu^{k-1}(d\bx)
		\\
		& \leq 2 
		\int_{(B_{kr}(x))^{k-1} \setminus (B_{\delta_1r}(x))^{k-1}} 
		\exp(-n \nu(B_r(x,\bx))) \nu^{k-1}(d \bx)
		\\
		& \leq 2 
		\int_{(B_{kr}(x))^{k-1} \setminus (B_{\delta_1r}(x))^{k-1}} 
		\exp(- n \nu(B_r(x)) -  f_0 \theta 
		\delta_2 n r^d) \nu^{k-1} (d \bx)
		\\
		& \leq c r^{d(k-1)} \exp(-n \nu(B_r(x)) 
		-f_0 \theta \delta_2 n r^d),
	\end{align*}
for some constant $c$ (independent of $x$).  Similarly using the simpler  bound 
$\nu(B_r(x,\bx)) \geq \nu(B_r(\bx))$ we obtain for $x \in A \setminus A^{(3r)}$
that $ \Pr[E_x] \leq c r^{d(k-1)} e^{-n\nu(B_r(x))}.$ Hence by splitting the
integral into regions $A^{(3r)}$ and $A \setminus A^{(3r)}$, given
$f_1^- < f_1$, and using Lemma \ref{l:bdy_est} we obtain for some new
constant $c$ that for $n$ large,
\begin{align}
	r^{d(1-k)} \int_A \Pr[E_x] \nu(dx) \leq c \exp(-  \theta 
	(f_0 + f_0\delta_2) n  r^d) + c r \exp(- \theta (f_1^-/2) n r^d).
		\label{e:intPE}
\end{align}

On the other hand, given $\delta_3 \in (0,\delta_1)$, and any $\eps >0$,
we have for large enough $n$ and all $x \in A^{(2r)}$, setting $B_s^+(x)$
to be the right half of $B_s(x)$, that
\begin{align*}
	\Pr[F_x] & \geq \int_{(B^+_{\delta_3 r}(x))^{k-1}} 
	(1- \nu(B_{(1+\delta_3)r}(x)))^{n-k} \nu^{k-1}(d \bx) 
\\
&  \geq (f_0 (\theta/2) \delta_3^d r^d)^{k-1} 
		\exp(-(1+ \eps)n \nu (B_{(1+\delta_3)r}(x))).
\end{align*}
Choose $\delta_3 $ so that $\fmax ((1+\delta_3)^d-1) \leq f_0 \delta_2/4$, and
	$\eps$ so that $2^d \eps \fmax < f_0 \delta_2/4$. Then
\begin{align*}
	(1+ \eps) \nu(B_{(1+\delta_3)r}(x)) & \leq \nu(B_r(x)) +
	\fmax \theta r^d((1+ \delta_3)^d - 1)) + \eps \fmax \theta (2r)^d
		\\
		&	\leq \nu(B_r(x)) + \theta r^d f_0  \delta_2/2, 
	\end{align*}
so that given $f_0^+ > f_0$ there is a further constant $c'>0$
	such that for large $n$ (using the continuity of $f$ on $A$)
	we have
	\begin{align*}
		r^{d(1-k)} \int_A \Pr[F_x] \nu(dx)
		\geq c' \exp(- (f_0^+ + f_0 \delta_2/2) \theta nr^d). 
	\end{align*}
	Therefore by (\ref{e:PrMc}) and (\ref{e:intPE})  there
	is a further constant $c''$  such that
	\begin{align*}
		\Pr[\MM^c] & \leq c''
		\exp[ 
		(f_0^+ + f_0 \delta_2/2 - f_0 - f_0 \delta_2)
		 \theta nr^d]
		+
		c'' r \exp[(f_0^+ + f_0 \delta_2/2 - f_1^{-}/2) \theta n r^d]
		\nonumber
		\\
		& = c'' \exp[(f_0^+ - f_0 - f_0 \delta_2/2) \theta n  r^d]
		+ c'' r \exp[(f_0^+ + f_0 \delta_2/2 - f_1^-/2) \theta n r^d].
	\end{align*}
	Taking $f_0^+$ sufficiently close to $f_0$
	and $f_1^-$ sufficiently close to $ f_1^-$,
	we have that the first
term is $O(\exp(-  f_0 (\delta_2/4) \theta n r^d))$, while taking
$b:= \limsup_{n \to \infty} ((n\theta r^d)/\log n)$,
the second term is bounded by $r n^{b(f_0- (f_1/2) +  f_0 \delta_2   )}$, 
	and by (\ref{e:de2}) and (\ref{e:supcri3}),
	this is $O(\exp (-c nr^d))$ for some $c>0$.
	 Thus using (\ref{e:N7comp}) and the fact that $N_6$
	 is bounded by a deterministic constant depending only
	 on $d$ and $k$, we have for some constant $c'$ that
	\begin{align*}
		\EE[N_6] & = \EE[N_6|\MM] \Pr[\MM]
		+ \EE[N_6|\MM^c] \Pr[\MM^c]
		\\
		& \leq \exp(-\delta_1(f_0/2) n r^d) + c' \exp(-cn r^d),
	\end{align*}
	which implies the result asserted.
\end{proof}

\subsection{Normal approximation}
\label{s:proof_CLT_k}

We shall prove our central limit theorem for $S'_{n,k}$
by expressing it as a sum of variables having the structure of
an
 {\em $m$-dependent random field},
 whose definition we now recall.  
Let $X=(X_{\alpha},  \alpha \in {\cal V})$ be a
collection of random variables indexed by a set ${\cal V} \subset \Z^d$. 
Given $m >0$,
we say $X$ is an {\em $m$-dependent random field}
if for any two subsets $A_1, A_2$ of $\Z^d$ with $\min_{\alpha \in A_1, 
\beta \in A_2}
\|\alpha -\beta\|_\infty >m$, the
sigma-algebras $\sigma\{
X_{\alpha}, \alpha \in A_1 \}$, and $\sigma\{ X_{\alpha}, \alpha
\in A_2 \}$, are mutually independent.

\begin{lemma} \label{l:ChenShao}  (see \cite[Theorem 2.6]{CS})
	Let $2 < q \leq 3$.  Let ${\cal V} \subset \Z^d$ and let
	$(W_i, \ i \in {\cal V})$ be an $m$-dependent
	random field with $\EE[W_i]=0$ for each $i \in {\cal V}$.
	Let $W =
\sum_{i \in {\cal V} } W_{i}.$  Assume  that $\E [W^2] = 1$
	and $\E[|W_i |^q] < \infty$ for all $i \in  {\cal V}$.
	Then 
	\begin{align}
 \label{CS1}
		\dk (W,N(0,1)) \leq 75 (10 m +1)^{(q - 1)d } 
		\sum_{i \in {\cal V}} \EE [|W_i|^q].
	\end{align}
\end{lemma}

We shall apply Lemma \ref{l:ChenShao} to $S'_{n,k}$ to prove (\ref{e:CLT_dk})
from Theorem \ref{t:int_k}.
The claim is restated in the form of a proposition. 

\begin{proposition}
\label{p:clt_iso_po}
	Set $b= \limsup_{n \to \infty} n \theta r^d/(\log n)$.
	Set $\sigma= (\Var[S'_{n,k}])^{1/2}$. Let $\eps >0$.
	Then as $n \to \infty$,
	\begin{align}
	\dk( \sigma^{-1}(S'_{n,k}-I_{n,k}),N(0,1)) = O(n^{\eps + (bf_0-1)/2 }). 
	\label{e:SnPoCLTK}
\end{align}
\end{proposition}
\begin{proof} 
	Let $f_0^-< f_0$ and $f_0^+ > f_0$, and $f_1^- < f_1^* < f_1$. 
	Fix $\delta \in (0,d^{-1/2})$, with $f_0(1-2 d \delta)^d > f_0^-$
	and also $f_1^* (1-2d \delta)^d > f_1^-$.
	Given $n$,
	partition $\R^d$ into cubes of side length $\delta r= \delta r(n)$
	centred on the
	points of $\delta r \Z^d$ and indexed by $\Z^d$; for $z \in \Z^d$
	let $C_z $ be the cube
	in the partition that is centred on $\delta rz$. Let
	${\cal V} = \{z \in \Z^d: C_z \cap A \neq \emptyset\}$.
	Then $|{\cal V}|= O(r^{-d})$ as $n \to \infty$.

	For $z \in {\cal V}$, let $U_z $ be the number of
	components of order $k$ in $G(\Po_n,r)$ having
	their left-most vertex in $C_z$, set $p_z:= \EE[U_z]$,
	and let $W_z = \sigma^{-1}
	(U_z - p_z)$. Then $\sum_{z \in {\cal V}}U_z  = S'_{n,k}$
	and  $\sum_{z \in {\cal V}} W_z =  
	\sigma^{-1} (S'_{n,k} - I_{n,k})$.
	Moreover, $(W_z,z \in {\cal V})$ is a centred $m$-dependent random
	field, where $m$ is independent of $n$. Hence  by Lemma 
	\ref{l:ChenShao} (taking $q=3$),
	there is a constant $c$ independent of $n$ such
	that
	\begin{align}
		\dk(\sigma^{-1}(S'_{n,k}-I_{n,k}),N(0,1))
		\leq
		c \sum_{z \in \nu} \E[|W_z|^3]
		= c \sigma^{-3} \sum_{z \in \nu} \E[|U_z - p_z|^3].
		\label{e:0830a}
	\end{align}
	Since $\delta < d^{-1/2}$, any two points of $\Po_n$ in the same
	cube are connected in $G(\Po_n,r)$, and therefore $U_z$ is
	a Bernoulli random variable with parameter $p_z$, for
	each $z \in \cV$. Hence
	$\E[|U_z-p_z|^3] = p_z(1-p_z)^3 + (1-p_z)p_z^3 \leq p_z$.

	Let ${\cal V}_1:= \{z \in \cV: \delta r z \in A^{(r)}\}$
	and $\cV_2:= \cV \setminus \cV_1$.

	Let $z \in \cV$.
	If $\Po_n ( B_{r-d \delta r} ( \delta r z)) \geq k+1$, then
	$U_z =0$.  Therefore there exists $c>0$ such that for $z \in \cV_1$,
	$$
	p_z \leq \Pr[\Po_n ( B_{r-d \delta r}(\delta r z) ) \leq k] 
	\leq c (nr^d)^k \exp(-  f_0 \theta (1- d\delta)^d n r^d ),
	$$
	so that by Proposition  \ref{p:var_k},
	there exists $c'$ such that for $n$ large 
	the contribution of $\cV_1$ to the right hand
	side of (\ref{e:0830a}) is at most
	$$
	c' r^{-d} \exp (-f_0 \theta (1-2d \delta)^d n r^d) I_{n,k}^{-3/2} 
	\leq c' r^{-d} \exp (-f_0^- \theta n r^d) I_{n,k}^{-3/2}, 
	$$
	and by Lemma \ref{l:Ilower}, for $n$ large this is at most 
	$$
	n^{-1/2} (nr^d)^{-1} \exp (  
	((3/2)f_0^+ - f_0^- ) \theta n r^d).
	$$
	Set $b := \limsup_{n \to \infty} n \theta r^d/(\log n)$.
	Then $b < 1/f_0$ by (\ref{e:supcri3}), and 
	for any $\eps >0$, provided $f_0^+$ and $f_0^-$ are
	 chosen close enough to
	$f_0$ the above is bounded by $n^{\eps +(bf_0 -1)/2 }$.

	Now consider $z \in \cV_2$. For such $z$, using
	Lemma \ref{l:bdy_est} we have
	$$
	p_z \leq \Pr[\Po_n ( B_{r-d \delta r}(\delta r z) ) \leq k] 
	\leq c (nr^d)^k \exp(-  f_1^* (1- d\delta)^d (\theta/2)  n r^d ).
	$$
	Also
	we claim that $|\cV_2| = O(r^{1-d})$.  
	Indeed, using the first part of Lemma \ref{l:bdy_est}, we can
	cover $\partial A$ by $O(r^{1-d})$ balls of radius $r$.
	Taking balls of radius $3r$ with the same centres gives a collection
	of $O(r^{1-d}) $ balls such that each each point of $\cV_2$ lies 
	in at least one of these balls. Since the number of
	points of $\cV_2$ contained in any one of these balls is uniformly
	bounded, the claim follows.

	Therefore there exists a constant $c''$ such that
	the contribution of $\cV_2$ to the right hand side
	of (\ref{e:0830a}) is at most
	$
	c'' r^{1-d} \exp(- f_1^- (\theta/2)  n r^d)
	I_{n,k}^{-3/2},
	$
	and by Lemma \ref{l:Ilower}, given $\eps >0$, for $n$
	large this is at most
	$$
	n^{-1/2} r (nr^d)^{-1} \exp(   ((3/2)f_0^+ - f_1^-/2) \theta nr^d)
	< n^{\eps -(1/2) - (1/d) + (b/2)(3f_0- f_1)}.  
	$$
	Using (\ref{e:supcri3}) we have that $b(f_0-f_1/2) < 1/d$
	and therefore the exponent of $n$ above is
	at most $\eps - (1/2) + b f_0/2$, i.e. $\eps + (bf_0-1)/2$.
	Combining this with the contribution of ${\cal V}_1$ and
	applying (\ref{e:0830a}) yields (\ref{e:SnPoCLTK}).
\end{proof}


For a binomial counterpart of Proposition \ref{p:clt_iso_po}
using Wasserstein distance, we use  the following lemma, which
is based on a result of Chatterjee \cite{Cha08}, which requires further
notation.

Recall the definition of $\bN(\R^d)$ from Section \ref{s:PoSprime}.
Given $F: \bN(\R^d) \to \R$ and $x \in \R$, set
	 $D_xF(\X):= F(\X \cup \{x\} ) - F(\X)$, and set
	$\|DF\| := \sup_{x \in \R^d,\X \in \bN(\R^d)}
	|D_xF(\X)|$.
	Given such $F$ and given $s \in (0,\infty)$, we say that
	the constant $s$ is
 a {\em radius of stabilization} for $F$ if
 $D_xF(\X) $ is determined by $\X \cap B_s(x)$
for all $x \in \R^d, \X \in \bN(\R^d)$, i.e.
$D_xF(\X) = D_xF (\X \cap B_s(x))$ for all $x,\X$.

\begin{lemma}
	\label{l:fromChat}
	Let $n \in \N$, $s \in (0,\infty)$. 
Suppose $F:\bN(\R^d) \to \R$ is measurable with
	$\|DF\| 
	< \infty $
	and $\sigma := \sqrt{\Var[F(\X_n)]} \in 
	(0,\infty)$, and $s$
is a radius of stabilization for $F$.
Then 
	\begin{align}
		\dw \left( \frac{F(\X_n) - \EE F(\X_n)}{\sigma},\cN \right)
		\leq \frac{ C n^{1/2}}{\sigma^2}  \|DF\|^2(\E[(\X_{n+4}(B_{s}
		(X_1)))^4])^{1/4}
		+ \frac{4n}{ \sigma^3}  \|DF\|^3,
		\label{e:Chat}
	\end{align}
	where $C$ is a universal constant.
\end{lemma}
\begin{proof}
We apply  \cite[Theorem 2.5]{Cha08}.  As  explained below, the graph
$G(\X_n,s)$ is a  {\em symmetric interaction rule} (in the sense of 
\cite{Cha08}) for $F$, and we can take its
	 symmetric extension $G'$ to be $G(\X_{n+4},s)$.
	The quantity denoted $\Delta_j f(X)$ in \cite[Theorem 2.5]{Cha08} 
	has its absolute value bounded by $2\|DF\|$.

	Suppose
	 $\bx = (x_1,\ldots,x_n) \in (\R^d)^n$ and 
	 $\bx' = (x'_1,\ldots,x'_n) \in (\R^d)^n$.
	 For $i,j \in \{1,\ldots,n\}$ with $i \neq j$, set  
	$\bx^{i}= (x_1,\ldots,x_{i-1},x'_i,x_{i+1},\ldots,x_n)$
	and $\bx^{ij}= (\bx^{i})^j$. As done elsewhere in this
	paper, we identify $\bx$ with the set $\{x_1,\ldots,x_n\}$
	(with any repeated entry in $\bx$ being included just once
	in $\{x_1,\ldots,x_n\}$),
	and likewise for $\bx^i$ and $\bx^{ij}$.

	By the stabilization condition $F(\bx) - F(\bx^j)$ is determined
	by $x_j, x'_j$ and the sets
	$\bx \cap B_s(x_j)$ and $\bx \cap B_s(x'_j)$. 
	Similarly,
	$F(\bx^i) - F(\bx^{ij})$ is determined
	by  $x_j, x'_j$ and the sets $\bx^i \cap B_s(x_j)$ and
	$\bx^i \cap B_s(x'_j)$. 

	If $\|x_i-x_j\|$, $\|x'_i-x_j\|$, $\|x_i-x'_j\|$
	and $\|x'_i -x'_j\|$ all exceed $s$ then  
	 $\bx \cap B_s(x_j) = \bx^i \cap B_s(x_j)$ and
	 $\bx \cap B_s(x'_j) = \bx^i \cap B_s(x'_j)$,
	so that $F(\bx)- F(\bx^j)= F(\bx^i)-F(\bx^{ij})$,
	and this shows that $G(\bx,s)$ really is a symmetric interaction
	rule for $F$ in the sense of \cite{Cha08}, as claimed earlier.
\end{proof}
We can now provide the binomial counterpart for Proposition 
\ref{p:clt_iso_po}
using
Wasserstein distance. This gives us the penultimate assertion
(\ref{e:CLTBi}) of Theorem \ref{t:int_k}.


\begin{proposition}
\label{p:clt_iso_bin}
	Set $\sigma= \sqrt{\Var[S_{n,k}]}$, and  set $b:= \limsup_{n \to \infty}
	n\theta r^d/(\log n) $. Then
\begin{align}
	\dw( \sigma^{-1}(S_{n,k}-\EE[S_{n,k}]),N(0,1)) & = O \left(
	\frac{n}{\sigma^3} \right) \nonumber \\
	& = O(n^{ \eps +(3b f_0  - 1)/2}).  \label{e:SnBiCLTW}
\end{align}
\end{proposition}


\begin{proof}
	Given $n$, we apply Lemma \ref{l:fromChat}, taking  $F(\X):=
	K_{k,r}(\X)$,
	the number of $k$-components in $G(\X,r)$ (with $r=r(n)$
	as usual). Then $\|DF\|$ is bounded above by a constant independent of
	$n$, since for
	all $x \in \R^d$, $r >0$ and finite $\X \subset \R^d$ the
	number of components of $G(\cX,r)$ intersecting $B(x,r)$
	is bounded by a constant depending only on $d$ (and not
	on $\X, x$ or $r$). Also $\|DF\| \geq 1$. By
	Propositions \ref{p:Snmeanbin}
	and
	\ref{p:var_k}, $\E[F(\X_n)] \sim I_{n,k}$ and 
	$\sigma^2 \sim I_{n,k}$ as $n \to \infty$.

	Clearly $s= (k+1)r$ is a radius of stabilization for $F$.
	Also $\X_{n+4}(B_{(k+1)r}(X_1))$ is stochastically dominated by 
	$1+\Bin(n+4,\fmax \theta ((k+1)r)^d)$ and hence $\EE[ \X_{n+4}
	(B_{(k+1)r}(X_1))^4] \leq
	c (nr^d)^4$, for some constant $c$. Therefore
	in the present instance, the first term in the right hand
	side of (\ref{e:Chat}) divided by the second term
	is $O(n^{-1/2}\sigma (nr^d)) = O(n^{-1/2} (nr^d)I_n^{1/2})$.
	which tends to zero because $I_n^{1/2}
	= O( n^{1/2} (nr^d)^{(k-1)/2} e^{-c nr^d})$ 
	for some $c>0$, by (\ref{e:exp_Snk}) and (\ref{e:volLB}).
	Then we obtain the first line of (\ref{e:SnBiCLTW}) from (\ref{e:Chat}).
	 
	 For the second line of (\ref{e:SnBiCLTW}) we use
Proposition \ref{p:var_iso_bin}, Lemma \ref{l:Ilower} and (\ref{e:supcri}).
\end{proof}

\begin{proof}[Proof of Theorem \ref{t:int_k}] 
	We have already proved
	(\ref{e:k_pois}),
	(\ref{e:k_bin}),
	(\ref{e:CLT_dk})
	and
	(\ref{e:CLTBi}).
	We now prove the normal approximation result (\ref{e:dKs}),
	using the Poisson approximation result (\ref{e:k_bin}).
	By the Berry-Esseen theorem $\dk(t^{-1/2}(Z_t-t),\cN)=
	O(t^{-1/2})$ as $t \to \infty$. Hence, by
		the triangle inequality for $\dk$,
		and the fact that $\dk(X,Y) \leq \dtv(X,Y)$ for any $X,Y$,
	we have
\begin{align}
	\dk( (\E S_{n,k})^{-1/2}(S_{n,k}- \E S_{n,k}),\cN)\le
	\dtv(S_{n,k},Z_{\E S_{n,k}}) + O((\E S_{n,k})^{-1/2}).
	\label{e:PotoNbin}
\end{align} 
By (\ref{e:k_bin})
	the first term in the right hand side of (\ref{e:PotoNbin})
	is $O(e^{-cnr^d})$ for some $c>0$.

By Proposition \ref{p:Snmeanbin} and Lemma \ref{l:Ilower}, 
$\E S_{n,k} = \Omega(n e^{-  nf_0^+\theta r^d})$.
		Therefore $(\E S_{n,k})^{-1/2} = o(n^{\eps+(bf_0-1)/2})$.
		Note that (\ref{e:supcri3}) implies
		$bf_0 <1$. Moreover, using (\ref{e:supcri3}) and
		taking both $\eps$ and $c$ to be sufficiently small,
		we see that
		the second term in the right hand side of 
		(\ref{e:PotoNbin}) is also
		$O(e^{-cnr^d})$, and hence the left hand
		side of (\ref{e:PotoNbin}) is $O(e^{-cnr^d})$ for some $c>0$.
		Thus
		\begin{align*}
		\dk((\Var S_{n,k})^{-1/2}(S_{n,k}- \E S_{n,k}),
			(\E S_{n,k}/\Var S_{n,k})^{1/2} \cN)
			\\
			= \dk( (\E S_{n,k})^{-1/2}
			(S_{n,k} - \E S_{n,k}) ,\cN ) =
			O(e^{-cnr^d}).
		\end{align*}
		Now using
		the fact that $\sup_{t \in [-1/2,1/2] \setminus \{0\}}
		t^{-1} \dk((1+t) \cN,\cN) < \infty$, and
		 Proposition \ref{p:var_iso_bin}, and the triangle
		inequality for $\dk$, we can
		 deduce
		(\ref{e:dKs}).

		The same argument works for $S'_{n,k}$, for all $d \geq 1$.
\end{proof}



\section{The sparse limiting regime}
\label{s:sparse}

Fix $k \in \N$.  In this section, instead of (\ref{e:supcri3}) we assume 
the `mildly sparse' limiting regime
\begin{align}
	\lim_{n \to \infty} nr_n^d = 0; ~~~~~~~~~~~~
	\lim_{n \to \infty} n (nr_n^d)^{k-1} = \infty.
	\label{e:subcri}
\end{align}
In Section \ref{subsecintro}, we mentioned how to obtain
a CLT for $S_{n,k}$ from previously known results in this regime under
the extra condition that $n(n r^{d})^k \to 0$.
We now describe, without giving full details, how we can adapt the methods of this paper to derive 
limiting expressions for means and variances, and CLTs
both for $S_{n,k}$ and $S'_{n,k}$, assuming only (\ref{e:subcri}).

We assume $\fmax < \infty$ but now we do not need to make any other assumptions
on $f$ or on the geometry of the support $A$  of $f$.
Defining $I_{n,k} := \E S'_{n,k}$ as before, we now have 
\bea
I_{n,k} \sim k!^{-1} n (nr^d)^{k-1} \int_{\R^d} f(x)^k dx \int_{(\R^d)^{k-1}}
h_1((o,\bx)) d \bx ~~~~~~~~~{\rm as} ~ n \to \infty.
\label{e:Insparse}
\eea
Moreover $\E[S_{n,k}] \sim I_{n,k}$ as $n \to \infty$.
These can be proved along the lines of
\cite[Propositions 3.1 and 3.2]{Pen03}.
By (\ref{e:Insparse}), (\ref{e:subcri}) implies $I_{n.k} \to \infty$ as 
$n \to \infty$. Under (\ref{e:subcri}), any factors of the form $e^{-n \nu(B_r(\bx))}$
arising in moment estimates no longer tend to zero, but remain bounded above
by 1.

Next we have
$$
\Var [S'_{n,k} ] = I_{n,k} (1+ O(nr^d)^k).
$$
To see this, follow the proof of Proposition \ref{p:var_k}.
In the sparse regime,  it can be seen directly from (\ref{e:J1def})
and (\ref{e:J2def})
that both $J_{1,n}$ and
$J_{2,n}$ are $O(n^{2k}r^{d(2k-1)})$, which is $O( (nr^d)^kI_{n,k})$.

 Next we have
 \bea
 \dtv (S'_{n,k} ,Z_{I_{n,k} } ) = O((nr^d)^k).
 \label{e:dtvPsparse}
 \eea
 For this, we can follow the proof of Proposition \ref{p:poi}.
 We now have $\E [ \xi_1(\bx)] = O((nr^d)^k)$, uniformly over $\bx$.
 Likewise $\E [ \xi_2(\bx)] = O((nr^d)^k)$.

For $S_{n,k}$, we consider only the case $k \geq 2$ in the sparse limiting
regime, since isolated vertices are not rare events.
We claim that if $k \geq 2$ then
\bea
\Var[S_{n,k}] = \E[S_{n,k}](1+ O(nr^d)^{k-1}).
\label{e:ratiosparse}
\eea
This is proved by following the proof of Proposition 
\ref{p:var_iso_bin}. The expression at (\ref{0729d}) is $O(n^{2k-1}r^{d(2k-2)})$
which is $O(I_{n,k} (nr^d)^{k-1})$. In the penultimate line of
(\ref{0730a}) the last factor (in square brackets) now simplifies 
to $O(n^{-1})$ so the expression in (\ref{0730a}) is $O(n^{2k-3})$.
Hence the first term in the right hand side of (\ref{0729e}),
multiplied by $n^2$, is now $O(n^{2k-1} r^{d(2k-2)})=
O((nr^d)^{k-1}I_{n,k})$. 
It can be seen directly that
the second and third terms in the right hand side of (\ref{0729e}),
multiplied  by $n^2$, are $O(n^{2k} r^{d(2k-1)})$
which is $O((nr^d)^{k} I_{n,k})$. 
Combining these estimates
shows that $\Var S_{n,k} - \E S_{n,k} = O((nr^d)^{k-1}I_{n,k})
 = O((nr^d)^{k-1} \E S_{n,k})$, as claimed.

 Next we claim that 
 by following the proof of Proposition \ref{p:PoApproxBi}, one can show
 \bea
 \dtv(S_{n,k},Z_{\E S_{n,k}})= O((nr^d)^{k-1}).
 \label{e:dtvsparse}
 \eea
 Indeed, for each $i \in [6]$ we have  $\E[N_i] = O((nr^d)^j)$
 with $j = j(i) \in \{k-1,k,k+1\}$. For  $N_6$ this can now be 
 seen directly without using the more involved
 proof of Lemma \ref{l:N7}. In particular, the restriction to $d \geq 2$
 is not needed for (\ref{e:dtvsparse}).

 By a similar argument to the proof of 
 (\ref{e:dKs}), using (\ref{e:dtvsparse}) and (\ref{e:ratiosparse}),
 we can obtain that
 $$
 \dk((\Var S_{n,k})^{-1/2}(S_{n,k}- \E S_{n,k}),\cN)
 = O(\max((nr^d)^{k-1}, (n (n r^d)^{k-1})^{-1/2})).
 $$
 Using (\ref{e:dtvPsparse}) instead of (\ref{e:dtvsparse}) we can similarly
 obtain that
 $$
 \dk((\Var S'_{n,k})^{-1/2}(S'_{n,k}- \E S'_{n,k}),\cN)
 = O(\max((nr^d)^{k}, (n (n r^d)^{k-1})^{-1/2})).
 $$


	{\bf Acknowledgements.} We thank Frankie Higgs and Oliver Penrose
	for some useful discussions in relation to this paper.


	This research was supported by Engineering and Physical Sciences Research Council (EPSRC) grant EP/T028653/1.


\bibliographystyle{plain}

\end{document}